\documentclass[10pt]{amsart}

\usepackage{float}

\usepackage{amsmath, amssymb, amsthm, color, fullpage}
\usepackage{graphicx, amsaddr}
\usepackage{enumerate}

\newtheorem{theorem}{Theorem}[section]
\newtheorem{corollary}[theorem]{Corollary}
\newtheorem{conjecture}[theorem]{Conjecture}
\newtheorem{lemma}[theorem]{Lemma}
\newtheorem{proposition}[theorem]{Proposition}
\newtheorem{definition}[theorem]{Definition}
\newtheorem{claim}[theorem]{Claim}
\newtheorem{case}{Case}[theorem]

\newtheorem{question}[theorem]{Question}

\usepackage{thmtools, thm-restate}

\title{Connectivity of contraction-critical graphs}

\author{Michael Lafferty$^{1}$, Runrun Liu$^{2}$,  Martin Rolek$^{3}$,    Gexin Yu$^{1}$}

\address{
$^{1}$\small Department of Mathematics, William \& Mary, Williamsburg, VA, 23185, USA.\\
$^{2}$\small School of Mathematics, Zhejiang Normal University, Jinhua, 321004, China.\\
$^{3}$\small Department of Mathematics, Kennesaw State University, Marietta, GA, 30060 USA.}

\thanks{The research of the second author was supported in part by NSFC, China (No. 12101563) and ZJNSFC, China (No. LQ22A010011). The second author is the corresponding author of this article.}

\email{mmlafferty@wm.edu, 827261672@qq.com, mrolek1@kennesaw.edu, gyu@wm.edu}

\date{\today}

\begin{document}
\maketitle

\begin{abstract}
Contraction-critical graphs came from the study of minimal counterexamples to Hadwiger's conjecture. A graph is $k$-contraction-critical if it is $k$-chromatic, but any proper minor is $(k-1)$-colorable. It is a long-standing result of Mader that $k$-contraction-critical graphs are $7$-connected for $k\ge7$. In this paper, we provide the improvement of Mader's result for small values of $k$. We show that $k$-contraction-critical graphs are $8$-connected for $k\ge17$, $9$-connected for $k\ge29$, and $10$-connected for $k\ge41$. As a corollary of one of our intermediate results, we also prove that every $30$-connected graph is $4$-linked.
\end{abstract}

\section{Introduction}
Graph coloring is a central topic in graph theory.  A graph is properly $2$-colorable if and only if it is bipartite.  For $k\ge 3$, it is NP-complete to decide whether a graph is properly $k$-colorable.  It is fascinating to know what reasons there may be for a graph to have high chromatic number.
Equally interesting is the problem of determining what sort of structures graphs of high chromatic number may contain.
In 1943, Hadwiger~\cite{H43} conjectured that every $k$-chromatic graph has a $K_k$-minor, where a graph $H$ is a minor of a graph $G$ if $H$ can be obtained from a subgraph of $G$ by contracting edges.  This conjecture is one of the deepest conjectures in graph theory.

It is not hard to show that Hadwiger's conjecture holds for $k\le 3$. Hadwiger~\cite{H43} and independently Dirac~\cite{D52} confirmed the case $k=4$. In 1937, Wagner~\cite{W37} proved that the case $k = 5$ is equivalent to the Four Color Theorem.
About 60 years later, Robertson, Seymour and Thomas~\cite{RST93} proved that the case $k = 6$ is also equivalent to the Four Color Theorem.
The conjecture remains open for $k\ge 7$.

A graph $G$ is said to be {\sl $k$-contraction-critical} if $G$ is $k$-chromatic, but any proper minor of $G$ is $(k - 1)$-colorable.   Hadwiger's conjecture is equivalent to the claim that the only $k$-contraction-critical graph is the complete graph $K_k$.
For their value in the study of Hadwiger's conjecture, the connectivity properties of noncomplete contraction-critical graphs have long been examined.
Let $h(k)$ be the largest integer such that every noncomplete $k$-contraction-critical graph is $h(k)$-connected.  Dirac~\cite{D60} initiated the study of connectivity of contraction-critical graphs in 1960 and proved that $h(k) \ge 5$ for $k \ge 5$.
In 1968, Mader~\cite{M68} extended this, and his following deep result has been extensively utilized in proving many results related to Hadwiger's conjecture, see~\cite{J71, KT05, RS17}.

\begin{theorem}[Mader~\cite{M68}]\label{thm:7Conn}
Non-complete $6$-contraction-critical graphs are $6$-connected, and non-complete $k$-contraction-critical graphs are $7$-connected for $k \ge 7$.
That is, $h(6) \ge 6$ and $h(k) \ge 7$ for $k \ge 7$.
\end{theorem}

More recent work has focused on improving $h(k)$ for large values of $k$.
Kawarabayashi~\cite{K07} proved the first general result, showing that $h(k) \ge \left\lceil\frac{2k}{27}\right\rceil$.
This was improved by Kawarabayashi and the fourth author~\cite{KY13}, who showed $h(k) \ge \left\lceil\frac{k}{9}\right\rceil$.
Despite this recent progress, it seems hopeless to extend these proofs to get even $h(k) \ge \left\lceil\frac{k}{2}\right\rceil$.

As it is extremely difficult to prove Hadwiger's conjecture, it makes sense to consider the following weaker version of the conjecture:

\begin{conjecture}\label{conj:conn}
Any noncomplete $k$-contraction-critical graph is $k$-connected.  That is, $h(k) \ge k$.
\end{conjecture}

Thus we see Conjecture~\ref{conj:conn} holds for $k \le 7$ and remains wide open for $k \ge 8$.
While Toft~\cite{T72} has shown that any $k$-contraction-critical graph is $k$-edge-connected, a similar generalization of Theorem~\ref{thm:7Conn} for vertex connectivity seems very difficult.
This motivates us to look for ways to improve known values of $h(k)$.

Our main result in this paper is stated below.  

\begin{theorem}\label{main2}
Let $G$ be a $k$-contraction-critical graph.  Then
\begin{itemize}
\item If $k\ge 17$, then $G$ is $8$-connected;
\item If $k\ge 29$, then $G$ is $9$-connected;
\item If $k\ge 41$, then $G$ is $10$-connected.
\end{itemize}
It follows that $h(k)\ge 8$ for $k\ge17$, $h(k)\ge 9$ for $k\ge29$ and $h(k)\ge 10$ for $k\ge41$.
\end{theorem}

Linkage structure plays an important role in both the study of graph minors and our proof of Theorem~\ref{main2}. For an integer $k \ge 2$, a graph $G$ is {\em $k$-linked} if for every $2k$ vertices $u_1, v_1, u_2, v_2, \ldots, u_k, v_k$, one can find $k$ internally disjoint paths $P_1, \ldots, P_k$ such that $P_i$ connects $u_i$ and $v_i$. Clearly, a $k$-linked graph is $k$-connected.
It has been an interesting problem to determine the function $g(k)$ such that $g(k)$-connected graphs are $k$-linked. Jung~\cite{J70} showed that a $4$-connected graph is $2$-linked, except when $G$ is planar and the vertices $u_1, u_2, v_1, v_2$ are on a face of $G$ in this order, implying $g(2) \le 6$.
Thomas and Wollan~\cite{TW08} showed that a $6$-connected graph on $n$ vertices with at least $5n - 14$ edges is $3$-linked, implying $g(3) \le 10$, and also proved that $g(k)\le 10k$ in general in \cite{TW05}. As a consequence of one of our intermediate results in the proof of Theorem~\ref{main2}, we are able to show the following for the next open case, $k = 4$.

\begin{theorem}\label{thm:4linked}
If $G$ is $30$-connected, then $G$ is $4$-linked. Consequently, $g(4)\le 30$.
\end{theorem}

Our proof of Theorem~\ref{main2} combines ideas from Mader~(\cite{M68}, 1968) and ~\cite{KY13, TW05} in 2005 and 2013.  In Section~\ref{proof-thm2}, we will outline the proofs in detail.  Here we would like to highlight a few things in our proofs.

One of the main ingredients of Mader's proof of Theorem~\ref{thm:7Conn} is the following theorem.

\begin{theorem}[Mader 1968~\cite{M68}]\label{thm:MaderS-3}
Suppose $G$ is a $(k + 1)$-contraction-critical graph.
If $S \subseteq V(G)$ with $|S| \le k$ and $\alpha(G[S]) \ge |S| - 3$, then $G - S$ is connected.
\end{theorem}

Mader commented that if the condition $|S|\le k$ in Theorem~\ref{thm:MaderS-3} could be strengthened to $|S|\le k+1$, then the result would imply the Four Color Theorem. In this article, we fully generalize Theorem~\ref{thm:MaderS-3}.  Mader's original proof of Theorem~\ref{thm:MaderS-3}~\cite{M68} is written in German, and we believe that our present paper is the first time a similar method appears in the literature in English.

\begin{restatable}{theorem}{firstmader}\label{main}
For integers $k\ge1, t\ge 3, k \ge (s + 2^{t - 1} - t)$, suppose $G$ is a $k$-contraction-critical graph. If $|S| \le s$ and $\alpha(G[S]) \ge |S| - t$, then $G-S$ is connected.
\end{restatable}

Theorem~\ref{main} immediately gives the following corollary. Note that Dirac~\cite{D60} showed that separating sets in a contraction-critical graph cannot be a clique.

\begin{corollary}\label{cor:EasyConn}
For $t \ge 6$, any $k$-contraction-critical graph is $t$-connected for $k \ge 2^{t - 4} + 2$.
\end{corollary}

\begin{proof}
Suppose $G$ is $k$-contraction-critical for some $k \ge 2^{t - 4} + 2 = (t - 1) + 2^{(t - 3) - 1} - (t - 3)$.
Then if $S$ is a separating set of $G$ with $|S| \le t - 1$, it follows from Theorem~\ref{main} that $\alpha(S) < |S| - (t - 3) \le (t - 1) - (t - 3) = 2$.  That is, $\alpha(S) = 1$ and $S$ forms a clique, a contradiction.
\end{proof}

In particular, Corollary~\ref{cor:EasyConn} implies that $h(k)\ge 8$ for $k\ge18$, $h(k)\ge 9$ for $k\ge34$ and $h(k)\ge 10$ for $k\ge66$, but we obtain better bounds with some more effort, still using our generalized Theorem~\ref{main}.

The paper is organized as follows. In Section~\ref{proof-thm2}, we give a proof for Theorem~\ref{main2} and Theorem~\ref{thm:4linked} with deferred proofs of some main lemmas. In Section~\ref{sec:independence}, we give a proof of  Theorem~\ref{main}. In Section~\ref{sec:N[v]}, we prove Theorem~\ref{lemma2.5}. In Section~\ref{finding}, we prove Lemma~\ref{knitted1}, and place the tedious parts in the appendices.  We finish the article with some closing remarks.

\section{Proof of  Theorem~\ref{main2} and Theorem~\ref{thm:4linked} (with deferred proofs)}\label{proof-thm2}

In this section we give an overview of the proof of Theorem~\ref{main2} .

Let $G$ be a $k$-contraction-critical graph with $k \ge k_0$, and suppose for a contradiction that $G$ has a minimum separating set $S$ with $|S| \le m-1$, where $(k_0,m)\in \{(17, 8), (29, 9), (41, 10)\}$.

As $|S|\le m-1$ and $k_0\ge (m-1)+2^{(m-4)-1}-(m-4)$ and $G-S$ is not connected, by Theorem~\ref{main}, $\alpha(G[S])< |S|-(m-4) = 3$. That is, $\alpha(G[S])\le 2$.

Thus we can partition $S$ into subsets $S_1, S_2, \ldots, S_t$ such that each $S_i$ is a maximal independent set with $|S_i| \le 2$.  Then for each $S_i, S_j$ with $i \ne j$, there is an edge with one end in $S_i$ and one end in $S_j$.  Let $G_1, G_2$ be subgraphs of $G$ such that $G_1 \cup G_2 = G$, $G_1 \cap G_2 = G[S]$, and $G_i \ne G[S]$ for $i \in \{1, 2\}$.

Suppose that both $(G_1, S)$ and $(G_2, S)$ are knitted, that is, we can find disjoint connected subgraphs $C_1, \ldots, C_t$ in $G_1$ such that $S_i \subseteq C_i$, and disjoint connected subgraphs $D_1, \ldots, D_t$ in $G_2$ such that $S_i \subseteq D_i$.   A graph $W$ is $(x_1, x_2, \dots, x_t)$-knitted if $(W, S)$ is knitted for every partition $\mathcal{P} = \{S_1, S_2, \dots, S_t\}$ of $S$ with $|S_i| = x_i\in \{1,2\}$ for all $i$. Let $S_i=\{s_i, t_i\}$ when $|S_i|=2$ and $S_i=\{s_i\}$ when $|S_i|=1$.  Note that $(2, \ldots, 2)$-knitted is same as $k$-linked, when there are exactly $k$ twos. Bollob\'as and Thomason~\cite{BT96} were the first to introduce and study knitted graphs.

Now in $G$, by contracting $C_i$ into a single vertex for each $i$, we obtain a graph $G_2'$, and by contracting $D_i$ into a single vertex for each $i$, we obtain a graph $G_1'$.  Since $G$ is $k$-contraction-critical, both $G_1'$ and $G_2'$ are $(k - 1)$-colorable.  Consider colorings of $G_1'$ and $G_2'$ so that they have the same colors on $S_i$.  Such colorings exist since, by our choice of the partition of $S$, the vertices obtained from $S_1, S_2, \dots, S_t$ by contraction induce a clique.  We can then combine and extend these colorings to a $(k - 1)$-coloring of $G$ by expanding the sets $S_1, S_2, \dots, S_t$, a contradiction.

Therefore, we may assume that $(G_1, S)$ is not knitted.

We claim that $G_1$ does not contain  a $(2,2,2,1)$-knitted subgraph when $m=8$, or a $4$-linked subgraph when $m=9$, or a $(2,2,2,2,1)$-knitted subgraph when $m=10$. For otherwise, let $L$ be such a knitted subgraph. Since $G$ is $(m-1)$-connected, we can find $m-1$ disjoint paths from $S$ to $L$, from which we get a $(m-1)$-subset $S' \subseteq V(L)$ and a corresponding partition of $S'$ into $S_1', \ldots, S_t'$.  In $L$, we can find disjoint connected subgraphs $C_1', \ldots, C_t'$ such that $S_i' \subseteq C_i'$.  Then we can find connected subgraphs $C_1, \ldots, C_t$ in $G_1'$ such that $S_i \subseteq C_i$. This contradicts that $(G_1, S)$ is not knitted.

We also claim that $G_1$ does not contain two disjoint $K_6$ subgraphs when $m=8$. Since $m-1=7$, the set $S_t$ is a singleton and belongs to at most one of these subgraphs, so let $L$ be a $K_6$ subgraph of $G_1$ that does not contain $S_t$.  We can similarly find six disjoint paths from $S - S_t$ to $L$ and obtain disjoint connected subgraphs $C_1, \ldots, C_t$ such that $S_i\subseteq C_i$, where $C_t = S_t$.  This contradicts that $(G_1, S)$ is not knitted.

However, we are able to show that such subgraphs do exist in relatively dense subgraphs.

\begin{lemma}\label{knitted1}
Let $z$ be a vertex and $H$ be the graph induced by $N[z]$ such that
\begin{equation}\label{eq000} 
\mbox{$H$ satisfies at least one of the following:}
\end{equation}
\begin{quote}
	\begin{enumerate}[(i)]
	\item $n(H)\le p$ and  $\delta(H)\ge \left\lfloor \frac{p}{2}\right\rfloor+1$. 
	\item $n(H) \le p-2$, $\delta(H) \ge \left\lfloor \frac{p}{2} \right\rfloor$, and $H$ has at most 2 (non-adjacent) vertices of degree $\left\lfloor \frac{p}{2} \right\rfloor$. 
	\item $n(H) \le p-4$ and $\delta(H) \ge \left\lfloor \frac{p}{2} \right\rfloor$.
	\end{enumerate}
  \end{quote}
Then
\begin{enumerate}[(a)]
\item if $p=42$,  then $H$ contains a $(2, 2, 2, 2, 1)$-knitted subgraph.
\item if $p=30$, then $H$ contains a $4$-linked subgraph.
\item if $p=18$,  then $H$ contains a $(2, 2, 2, 1)$-knitted subgraph. 
\end{enumerate}
\end{lemma}

Therefore, to reach a contradiction, we turn our attention to find dense subgraphs with property \eqref{eq000} in $G_1$ (with $p \ge 42, 30$ and $18$, respectively).  The following classic result by Dirac provides us further information on $G$.

\begin{lemma}[Dirac 1960~\cite{D60}]\label{lem:alphaNu}
If $G$ is $k$-contraction-critical, then $\alpha(G[N(u)]) \le d(u) - k + 2$ for any $u \in V(G)$.
\end{lemma}

As a consequence of Lemma~\ref{lem:alphaNu}, since $G$ is $k$-contraction-critical with $k\ge k_0$, then $\delta(G) \ge k-1$. It follows that $u \in V(G_1) - S$ satisfies $d(u) \ge k-1$.

We claim that for each $u\in G_1-S$, $d(u)\ge k+1$.  For otherwise, suppose that some $u\in V(G_1)-S$ has $d(u)\le k$.  Then $d(u)=k$ or $d(u)=k-1$. For the latter case, by Lemma~\ref{lem:alphaNu}, $\alpha(G[N(u)])\le (k-1)-k+2=1$, that is, it is a clique of order $k-1\ge k_0-1\ge 10$, which is of course a knitted subgraph.   Therefore $d(u)=k$ and $\alpha(G[N(u)])=2$. Let $H=N[u]$. If some vertex $v\in N(u)$ has degree $d(v)\le \lceil \frac{k}{2}\rceil+1$ in $H$, then $(N(u)-N[v])\cup\{u\}$ has independence number $1$ and thus is a clique, whose size is at least $(k+1)-(\lfloor \frac{k}{2}\rfloor+2)+1\ge9$; therefore, we have a $K_9$, which is $(2,2,2,2,1)$-knitted, $4$-linked, and $(2,2,2,1)$-knitted.  So each vertex in $N(u)$ has degree more than $\lceil \frac{k}{2}\rceil+1$. Let $n(H)=k+1$ and $\delta(H)\ge\lceil \frac{k}{2}\rceil+2$. Then we obtain a subgraph $H$ satisfies \eqref{eq000}, a contradiction.

For a given graph $L$, a pair $(A, B)$ is a {\sl separation} if $V(L) = A \cup B$ and there is no edge between $A - B$ and $B - A$.  The {\sl order} of a separation $(A, B)$ is $|A \cap B|$. If $S' \subseteq A$, then we say that $(A,B)$ is a {\sl separation of $(L,S')$}.  A separation $(A, B)$ of $(L,S')$ is {\sl rigid} if $(G[B], A \cap B)$ is knitted. For $T\subseteq V(L)$, let $\rho(T)$ be the number of edges with at least one endpoint in $T$.

\begin{definition}\label{def:massed}
Let $L$ be a graph and $S' \subseteq V(L)$. Then for any integer $p \ge 0$, $(L, S')$ is {\em $p$-massed} if
\begin{enumerate}
\item[(i)] $\rho(V(L) - S') >\frac{p}{2}|V(L) - S'|$, and
\item[(ii)] every separation $(A, B)$ of $(L, S')$ of order at most $|S'| - 1$ satisfies $\rho(B - A) \le \frac{p}{2}|B - A|$.
\end{enumerate}
\end{definition}

We observe that  $(G_1,S)$ is $(k+1)$-massed. In fact, (i) is obvious since each vertex in $V(G_1)-S$ has degree at least $k+1$, and (ii) is also clear since there is no separation of $(G_1,S)$ of order less than $|S|$ in $G$.

\begin{definition}\label{def:minimal}
Let $L$ be a graph and $S' \subseteq V(L)$. For integers $l$ and $p$ with $l\le \lfloor \frac{p}{2} \rfloor-1$, the pair $(L,S')$ is $p$-minimal if
\begin{enumerate}
\item $(L, S')$ is $p$-massed,
\item $|S'| \le l$ and $(L,S')$ is not knitted,
\item subject to (1)-(2), $|V(L)|$ is minimum,
\item subject to (1)-(3), $\rho(V(L) - S')$ is minimum.
\item subject to (1)-(4), the number of edges in $L[S']$ is maximum.
\end{enumerate}
\end{definition}

We will prove the following result in Section~\ref{sec:N[v]}.  This result is essentially a restatement of Theorem 1.4 of Thomas and Wollan~\cite{TW05}, but there is a small gap in their proof (not important to their result though), and we actually can only get a slihgtly weaker one.

\begin{restatable}{theorem}{massedgraphs}\label{lemma2.5}
Let $p\ge 0$ be an integer.
Let $L$ be a graph and $S' \subseteq V(L)$ such that $(L, S')$ is $p$-minimal. Let $\alpha(G[N(S)])\le2$.
Then $L$ has no rigid separation of order at most $|S'|$, and $L$ has a vertex $v \notin S'$ such that the subgraph $H$ induced by $N[v]$ satisfies  at least one of the following:  (i) $n(H)\le p$ and $\delta(H) \ge \left\lfloor \frac{p}{2}\right\rfloor+1$; (ii) $n(H)\le p-2$, $\delta(H) \ge \left\lfloor \frac{p}{2}\right\rfloor$, and $H$ has at most two (non-adjacent) vertices of degree $\left\lfloor \frac{p}{2}\right\rfloor$; (iii) $n(H)\le p-4$ and $\delta(H)\ge \left\lfloor \frac{p}{2}\right\rfloor$.
\end{restatable}

Among all $(k+1)$-massed pairs $(G_1,S)$, we consider a minimal pair $(G_1',S)$ with  $(l, p) = (m-1, k+1)$.  By Theorem~\ref{lemma2.5}, $G_1'$ has no rigid separation of order at most $m-1$, and we can find a subgraph $H$ induced by $N[v]$ for some $v \in V(G_1')-S$ satisfies \eqref{eq000} (with $\ell=(k+1)-2\ge k_0-1$).  By Lemma~\ref{knitted1}, $H$ contains knitted subgraph $H_0$. Since $(G_1',S)$ has no rigid separation of order at most $|S|$, we can find $|S|$ disjoint paths from $S$ to $H_0$, thus $(G_1',S)$ is knitted, a contradiction. This completes the proof of Theorem~\ref{main2}.

\

We are now able to provide a quick proof of Theorem~\ref{thm:4linked}.
\medskip

{\em Proof of Theorem~\ref{thm:4linked}.}
Let $G$ be $30$-connected, and let $S \subseteq V(G)$ with $|S| = 8$ be arbitrary.
Then $\delta(G) \ge 30$, so $(G, S)$ is $30$-massed since $G$ has no separation of order at most $|S| - 1$.
It follows that $G$ has a subgraph $G'$ such that $S \subseteq V(G')$ and $(G', S)$ is minimal.  By Lemma~\ref{lemma2.5} and  Lemma~\ref{knitted1}, $H$ contains a $4$-linked subgraph, say $L$. 
Then $L$ is a subgraph of $G$, so there exist $8$ disjoint paths with one end in $S$, the other end in $L$, and no internal vertex in $L$.  As $L$ is $4$-linked, it follows that we can link the vertices of $S$ as desired.  Therefore, $G$ is $4$-linked.
\hfill$\square$

\section{Proof of Theorem~\ref{main}}\label{sec:independence}

In this section, we prove Theorem~\ref{main}.

\firstmader*

For shortness, Let $U \subseteq V(G)$.  A coloring $\phi$ of $G$ is {\sl $U$-monochromatic} if $\phi$ assigns the same color to every vertex of $U$.  If $\phi'$ is a coloring of the graph obtained from $G$ by contracting $U$ to a single vertex, then, when we say {\sl $\phi'$ can be extended to a coloring $\phi$ of $G$ by expanding the set $U$,} we mean that $\phi(v) = \phi'(v)$ for all $v \in V(G) - U$, and $\phi$ assigns to every vertex of $U$ the same color that $\phi'$ assigns to the contracted vertex.  Note that the coloring $\phi$ is a proper coloring of $G$ if $U$ is an independent set.

Suppose Theorem~\ref{main} is not true.
Let $t$ be maximal such that the result holds for $t - 1$ but does not hold for $t$.
Then by Theorem~\ref{thm:MaderS-3}, we have $t \ge 4$.
Suppose, for some $k \ge (s + 2^{t - 1} - t)$, $G$ is a $k$-contraction-critical graph with a separating set $S$ such that $|S| \le s$ and $\alpha(S) \ge |S| - t$.
By the choice of $t$, we may assume $\alpha(S) = |S| - t$.
Let $U \subseteq S$ be an independent set of order $|S|-t$ and let $W = S - U$.
Let $G_1$ and $G_2$ be subgraphs of $G$ such that $G_1 \cup G_2 = G$ and $G_1 \cap G_2 = G[S]$.
Let $r = k - 1 \geq s + 2^{t - 1} - t - 1$.
Let $\phi'$ be an $r$-coloring of the graph obtained from $G$ by contracting $G_2 - W$ to a single vertex.
Then $\phi'$ may be extended to a $U$-monochromatic $r$-coloring $\phi_1$ of $G_1$ by expanding the set $U$.
Since $U$ is a maximum independent set in $S$, the colors assigned by $\phi_1$ to the vertices of $W$ are distinct from the color assigned to the vertices of $U$.
Similarly, there exists a $U$-monochromatic $r$-coloring $\phi''$ of $G_2$.
Without loss of generality, we may assume that the number of colors used by $\phi_1$ on $W$ is at most as many colors used by $\phi''$ on $W$.
If $\phi_1$ assigns a distinct color to each vertex of $W$, then it is possible to permute the colors of $\phi''$ so that $\phi_1$ and $\phi''$ agree on $S$.
Then we may combine the colorings $\phi_1$ and $\phi''$ to obtain an $r$-coloring of $G$, a contradiction.

Therefore, we may assume that $\phi_1$ assigns the colors $\{1, 2, \dots, p\}$ to the vertices of $W$, where $p < |W|$, and no other $U$-monochromatic $r$-coloring of $G_1$ assigns more colors to $W$.
We will also assume that every vertex of $U$ is assigned the color $r$.
For $i \in \{1, 2, \dots, p\}$, let $V_i$ be the vertex set of $W$ assigned color $i$ by $\phi_1$.
We may assume $|V_1| \ge 2$.

To each set $V_i$ we now assign a list of colors $L_i$ satisfying the properties that $i \in L_i$, $r \notin L_i$, $i \notin L_j$ for all $i \ne j$, and given any subset $J \subseteq \{1, 2, \dots, p\}$ there exists a common color on each list $L_i$ with $i \in J$ that does not appear on any list $L_i$ with $i \notin J$.
In other words, we assign a unique color to each element of the power set of $\{V_1, V_2, \dots, V_p\}$ (except the empty set), and this color is added to the corresponding lists $L_i$ of all sets $V_i$ in that element of the power set.

Note that $|W| = |S| - |U|=t$ and $|V_1|\ge 2$, so $p \le t - 1$.
Thus $${p\choose 1}+{p\choose 2}+\ldots{p\choose p}=2^p - 1 \le 2^{t - 1} - 1$$ distinct colors have been assigned across all of the lists $L_i$.
If there exists $i \ge 2$ such that $|V_i| \ge 2$, then we assign an additional unique color to each list $L_i$ such that $|V_i| \ge 2$ for $i \in \{1, 2, \dots, p\}$.
If we add $q$ additional colors in this way, we must have $p \le t - q$, so we assign at most $2^p - 1 + q \le 2^{t - q} - 1 + q$ colors on all lists.
Since $q \ge 2$ and $t \ge 4$, we have $2^{t - q} - 1 + q \le 2^{t - 1} - 1$, so in any case at most $2^{t - 1} - 1$ colors are used on the lists $L_i$.

Consider the subgraph of $G_1$ induced by all vertices assigned colors of $L_1$ by $\phi_1$.
Then there must be a single component $C_1$ of this subgraph which contains all vertices of $V_1$.
Otherwise, we would be able to swap color $1$ with any other color of $L_1$ on a component which contains a vertex of $V_1$ in order to obtain a $U$-monochromatic $r$-coloring of $G_1$ with $p + 1$ colors on $W$, a contradiction.
Now let $i \in \{1, 2, \dots, p - 1\}$ be maximal such that the component $C_i$ has been chosen.
Consider the subgraph of $G_1 - (\cup_{j = 1}^i C_j)$ induced by all vertices assigned colors of $L_{i + 1}$ by $\phi_1$.
Again, there must be a single component $C_{i + 1}$ of this subgraph which contains all vertices of $V_{i + 1}$.
If $|V_{i + 1}| = 1$, this is obvious since the color $i + 1$ is unique to $L_{i + 1}$.
If $|V_{i + 1}| \ge 2$, then this follows by the same color swap argument as above when swapping the two colors unique to $L_{i + 1}$.
Thus we have recursively defined disjoint, connected subgraphs $C_1, C_2, \dots, C_p$ of $G_1$ such that $V_i \subseteq C_i$ for all $i$.

Now let $D_1, D_2, \dots, D_m$ be the components of $G_1 - (\cup_{i = 1}^p C_i)$.  
Let $\phi_2$ be an $r$-coloring of the graph obtained from $G$ by contracting $C_1, C_2, \dots, C_p, D_1, D_2, \dots, D_m$ each to a single vertex, and let $\phi_2'$ be the $r$-coloring of $G_2$ obtained from $\phi_2$ by expanding the sets $C_1 \cap S, C_2 \cap S, \dots, C_p \cap S, D_1 \cap S, D_2 \cap S, \dots, D_m \cap S$.
Note that for any $i$, all vertices of $V_i$ are assigned the same color by $\phi_2'$.
Let $W_1, W_2, \dots, W_{p'}$ be a minimal partition of $\{V_1, V_2, \dots, V_p\}$ such that for each $i$, all vertices of $W_i \cap (\cup_{j = 1}^p V_j)$ are assigned the same color by $\phi_2'$.  
For each $i$, the lists $L_j$ corresponding to the sets $V_j \in W_i$ have a common color which does not appear on any list $L_j$ corresponding to $V_j \notin W_i$.
We may assume that all vertices of the sets $V_j \in W_i$ are assigned this common color by $\phi_2'$.
If there are two common colors, then $W_i = \{V_j\}$ for some $V_j$ with $|V_j| \ge 2$, and in this case we assume the vertices of $V_j$ are assigned color $j$ by $\phi_2'$.
Since there are at least $r - |U| = r - (|S| - t) \ge 2^{t - 1} - 1$ colors not used by $\phi_2'$ on the vertices of $U$, we may assume that any color in $\{1, 2, \dots, 2^{t - 1} - 1\}$ which is not used by $\phi_2'$ on the vertices of $W$ is also not used on the vertices of $U$.
We now obtain an $r$-coloring $\phi_1'$ of $G_1$ from $\phi_1$ by performing the following color swaps.

\emph{(i) For $i \in \{1, 2, \dots, p\}$, if the vertices of $V_i$ are assigned the color $\lambda$ by $\phi_2'$, then we swap the colors $\lambda$ and $i$ on $C_i$.}

\emph{(ii) For $i \in \{1, 2, \dots, m\}$, if the vertices of $D_i \cap S$ are assigned the color $\lambda$ by $\phi_2'$, then we swap the colors $\lambda$ and $r$ on $D_i$.}

If $C_i$ is assigned the color $\lambda$, then $C_i$ is not adjacent to any other component $C_j$ also assigned the color $\lambda$.
By the choice of the colors $\lambda$ and $i$, and the construction of the component $C_i$, no neighbor of $C_i$ is assigned color $i$ or $\lambda$ by $\phi_1$.
Thus swapping the colors $\lambda$ and $i$ on $C_i$ still gives a proper $r$-coloring of $G_1$.
Similarly, if $D_i$ is assigned the color $\lambda$, then by construction of the components $C_j$, $D_i$ is not adjacent to any vertex of color $\lambda$ or $r$.
If $\lambda \in \{1, 2, \dots, p\}$, this follows from the fact that some component $C_j$ must also have been assigned the color $\lambda$.
Thus swapping the colors $\lambda$ and $r$ on $D_i$ also gives a proper $r$-coloring of $G_1$.
Therefore, $\phi_1'$ is a proper $r$-coloring of $G_1$ which now agrees with $\phi_2'$ on $S$.
These colorings can be combined to give a proper $r$-coloring of $G$, a contradiction.

\section{Finding a Dense Neighborhood: a proof of Theorem~\ref{lemma2.5}}\label{sec:N[v]}

We will prove Theorem~\ref{lemma2.5} in a sequence of claims.

\massedgraphs*

\begin{claim}\label{claim1}
$(G,S)$ has no rigid separation of order at most $|S|$.
\end{claim}

The proof is the same as that in \cite{TW05}. We include it here for completeness.

\begin{proof}
For otherwise, take a rigid separation $(A,B)$ of $(G,S)$ with $A$ minimum.

We first assume that $|A \cap B| \le l - 1$.
Let $G^*$ be the graph obtained from $G$ by adding all missing edges in $A \cap B$.
Consider $(G^*[A], S)$.
If $(G^*[A], S)$ is also massed, then $(G^*[A], S)$ is knitted by the minimality of $(G, S)$, and a knit in $G^*[A]$ can be easily converted into a knit in $(G, S)$ as follows.
Since $A \cap B$ is complete in $G^*[A]$, we may assume that each connected subgraph in the knit uses at most one edge with both ends in $A \cap B$, and edges of $E(G^*[A]) - E(G)$ may be replaced by a connected subgraph in $G[B]$ because $(A, B)$ is rigid.
Since $(G, S)$ is not knitted, we conclude that $(G^*[A], S)$ is not massed.
Since $(G, S)$ is massed, $\rho(V(G) - S) \ge \frac{p}{2}|V(G) - S|$ and $\rho(B - A) < \frac{p}{2}|B - A|$, hence $\rho(V(G^*[A]) - S) \ge \rho(V(G) - S) - \rho(B - A) > \frac{p}{2}|V(G) - S| - \frac{p}{2}|B - A| = \frac{p}{2}|V(G^*[A]) - S|$.
So $(G^*[A], S)$ satisfies (i), and thus does not satisfy (ii) in Definition~\ref{def:massed}.
Let $(A', B')$ be a separation of $(G^*[A], S)$ violating (ii) such that $S \subseteq A'$ and $B'$ is minimal.
Since $A \cap B$ forms a clique in $G^*[A]$, either $A \cap B \subseteq A'$ or $A \cap B \subseteq B'$.
If $A \cap B \subseteq A'$, then $(A' \cup B, B')$ is a separation in $G$ violating (ii), contradicting that $(G, S)$ is massed.
So $A \cap B \subseteq B'$.
Consider $(G^*[B'], A' \cap B')$.
The minimality of $B'$ implies that $(G^*[B'], A' \cap B')$ satisfies (ii), and $\rho(B' - A') \ge \frac{p}{2} |B' - A'|$ means that it satisfies (i) as well.
Thus $(G^*[B'], A' \cap B')$ is knitted by the minimality of $(G, S)$.
Then $(G^*[B \cup B'], A' \cap B')$ is knitted, which means that $A' \cap B'$ is a rigid separation of $(G, S)$, a contradiction to the minimality of $A$.

Now assume that $|A \cap B| = l$.
If there exist seven disjoint paths from $S$ to $A \cap B$, then the paths together with the rigidity of $(A, B)$ show that $(G, S)$ is knitted, a contradiction.
Thus there is a separation $(A'', B'')$ of $(G[A], S)$ of order at most $6$ with $A \cap B \subseteq B''$.
Choose such a separation with $|A'' \cap B''|$ minimum.
Then there are $|A'' \cap B''|$ disjoint paths from $A'' \cap B''$ to $A \cap B$, from the rigidity of $(A, B)$ we have $(A'', B \cup B'')$ is a rigid separation of $(G, S)$ with $|A''| < |A|$, a contradiction to the minimality of $A$.
\end{proof}

Note that $\alpha(G[N(S)])\le2$. So $S$ can be partitioned into $S_1, \ldots, S_t$ so that $S_i=\{s_i, t_i\}$ (when $|S_i|=1$ then $s_i=t_i$). Since $(G,S)$ is not knitted, condition (5) in Definition 2.4 implies that for some choice of the partition $S_1, \ldots, S_t$ of $S$, all pairs of vertices of $S$ are adjacent, except possibly the pairs $s_i, t_i$. Thus we may assume that the chosen partition of $S$ has this property.

\begin{claim}\label{claim2}
 Let $u,v$ be adjacent vertices of $G$ and at least one of them does not belong to $S$. Then $u$ and $v$ have at least $\lfloor \frac{p}{2}\rfloor-\epsilon$ common neighbors, where $\epsilon\in \{0,1\}$ with $\epsilon=1$ when one of $u$ and $v$ is in $\{s_i, t_i\}$ for some $i$, and the other is adjacent to both $s_i$ and $t_i$. Consequently, in $G[N[v]]$ for $v\not\in S$, all vertices not in $S$ has degree at least $\lfloor \frac{p}{2}\rfloor+1$, and each vertex in $S$ has degree at least $\lfloor \frac{p}{2}\rfloor$.
\end{claim}

\begin{proof}
Consider the graph $G' = G/uv$, the graph obtained from $G$ by contradicting the edge $uv$.
If $(G', S)$ is knitted, then $(G, S)$ is knitted.
Thus $(G', S)$ is not massed by the minimality of $(G, S)$, and so it violates either (i) or (ii) in Definition~\ref{def:massed}.

Assume first that $(G', S)$ violates (ii).
Let $(A', B')$ be a separation of $G'$ violating (ii) with $B'$ minimal.
Then $\rho(G'[B'] - A') \ge \frac{p}{2}|B' - A'|$, and in particular $(G'[B'], A' \cap B')$ is massed by the choice of $B'$.
By the minimality of $(G, S)$, the pair $(G'[B'], A' \cap B')$ is knitted.
So $(A', B')$ is a rigid separation of $(G', S)$ of order at most $l - 1$.
Note that the separation $(A', B')$ induces a separation $(A, B)$ in $G$, where we replace the contracted vertex of $G'$ with both $u$ and $v$.
If $\{u, v\} \not\subseteq A \cap B$, then $(A, B)$ is a rigid separation of $(G, S)$ of order at most $l - 1$, which is a contradiction to Claim~\ref{claim1}.
So we assume that $\{u,v\} \subseteq A \cap B$.
By the minimality of $B'$, $(G[B], A \cap B)$ satisfies (ii).
Since $\rho(G[B] - A \cap B) = \rho(G[B] - A) \ge \rho(G'[B'] - A') \ge \frac{p}{2}|G'[B'] - A'| = \frac{p}{2}|G[B] - A|$, we see $(G[B], A \cap B)$ satisfies (i), so it is massed and thus knitted.
Hence $(A, B)$ is a rigid separation of size at most $|A' \cap B'| + 1 \le l$, a contradiction to Claim~\ref{claim1} again.

So we may assume that $(G',S)$ violates (i).
Then
\[ \rho(V(G') - S) \le \frac{p}{2}|V(G') - S| = \frac{p}{2}|V(G) - S| - \frac{p}{2}< \rho(V(G) - S) - \frac{p}{2}. \]

As one of $u, v$ is not in $S$, edges that are counted in $\rho(V(G) - S)$ but not in $\rho(V(G') - S)$ include the following: the edge $uv$, one of $wu$ and $wv$ when $w$ is adjacent to both $u$ and $v$, and $vt_i$ when $u=s_i$. So $\rho(V(G') - S)=\rho(V(G) - S) - 1 - r-\epsilon$, where $r$ is the number of common neighbors of $u$ and $v$, and $\epsilon=1$ if $vt_i\in E(G)$ and $u=s_i$ and $\epsilon=0$ otherwise.  It follows that $r > \frac{p}{2} - 1-\epsilon$.
Hence $u$ and $v$ have at least $\left\lfloor\frac{p}{2}\right\rfloor-\epsilon$ common neighbors, and when $\epsilon=1$, $v$ is adjacent to both $s_i$ and $t_i$ and $u=s_i$ for some $i$.
\end{proof}

\begin{claim}\label{claim3}
$\rho(V(G) - S) \le \frac{p}{2}|V(G) - S| + 1$.
\end{claim}

\begin{proof}
Consider the graph $G-e$ for some edge $e\in E(G)$ which does not have both ends in $S$. If $(G-e,S)$ is $p$-massed, then by the minimality of $(G,S)$ the pair $(G-e,S)$ is knitted, and consequently, $(G,S)$ is knitted as well, a contradiction. Thus $(G-e,S)$ is not $p$-massed, and so fails (i) or (ii). If $(G-e, S)$ fails (ii), then $(G-e, S)$ contains a separation $(A, B)$ with $|A \cap B| \le l - 1$.
It follows that $u \in A - B$ and $v \in B - A$, since otherwise $(A, B)$ is a separation in $(G, S)$ violating (ii). Then $|N(u) \cap N(v)| \le |A \cap B| \le l - 1$. By Claim~\ref{claim2}, $|N(u) \cap N(v)| \ge \left\lceil\frac{p}{2}\right\rceil-1$.
So $\left\lceil\frac{p}{2}\right\rceil -1\le l - 1 \le \left\lceil\frac{p}{2}\right\rceil - 2$, a contradiction. Therefore $(G-e, S)$ fails (i), that is, $\rho(V(G-e) - S) \le \frac{p}{2}|V(G-e) - S|$. So $\rho(V(G) - S) \le \frac{p}{2}|V(G) - S| +1$.
\end{proof}

\begin{claim}\label{Claim4}
Let $\delta^*$ be the minimum degree in $G$ among the vertices of $V(G)-S$. Then $\delta^*<p$. 
\end{claim}

\begin{proof}
For $x\in S$ let $f(x)$ be the number of neighbors of $x$ in $V(G)-S$. Clearly, $f(x)\ge 1$, otherwise $(S, V(G)-x)$ is a separation of $(G,S)$ violating (2).  Then by Claim~\ref{claim2}, $f(x)\ge \lfloor \frac{p}{2}\rfloor-1-(l-2)+1\ge 3$. If $\delta^*\ge p$, then from Claim~\ref{claim3},
\begin{align*}
   p|V(G)-S|+2\ge 2\rho(V(G)-S)&=\sum_{v\in V(G)-S} d(v)+\sum_{x\in S} f(x) \ge p|V(G)-S|+3|S|,
\end{align*}
a contradiction, because $S\not=\emptyset$.
\end{proof}

Let $T$ be the set of vertices in $G-S$ with degree at most $p-1$. For each $v\in T$, let $H_v=G[N[v]]$. Then $n(H_v)\le p$ and $d_{H_v}(u)\ge \delta(H_v)\ge \lfloor \frac{p}{2}\rfloor-\epsilon+1$.  If the minimum degree of $H_v$ is at least $\lfloor \frac{p}{2}\rfloor+1$, then we obtain $H$ with $$n(H)\le p \text{ and }  \delta(H)\ge \lfloor \frac{p}{2}\rfloor+1.$$

We may assume that some vertices in $H_v$ have degree $\lfloor \frac{p}{2}\rfloor$. Then by Claim 4.2, they are in $S$, and moreover, if $v\in T$ has exactly two neighbors in $S$ (namely, $s_i$ and $t_i$),  then at most two vertices in $H_v$ have degree $\lfloor \frac{p}{2}\rfloor$ and the rest has degree at least $\lfloor \frac{p}{2}\rfloor+1$.  Let $T_1\subseteq T$ be the set of vertices $v\in T$ so that $H_v$ contains at most two vertices of degree $\lfloor \frac{p}{2}\rfloor$, and $T_2=T-T_1$. It implies that for $v\in T_2$, $H_v$ contains more than two vertices of degree $\lfloor \frac{p}{2}\rfloor$, thus $v$ is adjacent to at least four vertices (two pairs) of $S$.   If $d(v)\le p-3$ for $v\in T_1$ or $d(v)\le p-5$ for $v\in T_2$, then we obtain an $H$ with at most two (non-adjacent) vertices with degree $\lfloor \frac{p}{2}\rfloor$ and $$n(H)\le p-2 \text{ and }  \delta(H)\ge \lfloor \frac{p}{2}\rfloor.$$
or an $H$ with  $$n(H)\le p-4 \text{ and }  \delta(H)\ge \lfloor \frac{p}{2}\rfloor.$$
So we assume that for each $v\in T_1$, $d_G(v)\ge p-2$ and  $\delta(H_v)=\lfloor \frac{p}{2}\rfloor$, and for each $v\in T_2$, $d_G(v)\ge p-4$ and  $\delta(H_v)=\lfloor \frac{p}{2}\rfloor$.
For $x\in S$ let $f(x)$ be the number of neighbors of $x$ in $V(G)-S$.

Note that every vertex in $T_1$ is adjacent to at least two vertices in $S$, and every vertex in $T_2$ is adjacent to at least four vertices in $S$; let $E'$ be the set of $2|T_1| + 4|T_2|$ edges we obtain this way.
We claim that there are at least 3 edges with one end in $S$ that do not belong to $E'$.
Note that the edges of $E'$ come in pairs, so that, for any pair of the form $\{ s_i, t_i \}$ and any $y \in T$, $s_i y \in E'$ if and only if $t_i y \in E'$.
If, for any pair $\{ s_i, t_i \}$ in $S$, each of $s_i$ and $t_i$ is incident to at most 1 edge in $E'$ (note that this is automatically the case if $s_i = t_i$, since no unpaired vertex in $S$ is incident to any edge of $E'$), then, arguing as in Claim~\ref{Claim4}, we have $f(s_i) \geq \left\lfloor \frac{p}{2} \right\rfloor - 1 - (l - 2) + 1 \ge 3$ and $f(t_i) \ge 3$, giving us 4 edges outside of $E'$.
If, for any two pairs $\{ s_i, t_i \}$ and $\{ s_j, t_j \}$ in $S$, each of the vertices $s_i, t_i, s_j, t_j$ is incident to at most 2 edges of $E'$, so, since $\min \{ f(s_i), f(t_i), f(s_j), f(t_j) \} \ge 3$, we again have 4 edges that do not belong to $E'$.
Thus we may assume that $S$ has at most one pair $\{ s_i, t_i \}$ such that each of $s_i$ and $t_i$ is incident to exactly 2 edges of $E'$, with every other pair of vertices in $S$ being incident to at least 3 edges in $E'$.
For every set $S_0 = \{ s_{i_1}, t_{i_1}, \dotsc, s_{i_m}, t_{i_m} \}$ of $m$ pairs in $S$, consider the set $T_0 = \{ y \in T : y s \in E' \text{ for some } s \in S_0\}$.
Note that $|T_0| \geq m$: the number of edges in $E'$ with an end in $S_0$ is at least $3(2m)-2 = 6m-2$, so, since every vertex in $T_0$ is incident to at most 4 edges of $E'$, we have $|T_0| \geq \frac{6m-2}{4} \geq m$.
Then, by Hall's marriage theorem, every pair of vertices $\{ s_i, t_i \}$ has a distinct common neighbor in $T$, so that $(G,S)$ is knitted and therefore not minimal, a contradiction.

By Claim~\ref{claim3},  we have
\begin{align*}
    p|V(G)-S|+2&\ge 2\rho(V(G)-S)=\sum_{v\in V(G)-T-S} d(v)+\sum_{v\in T}d(v)+\sum_{x\in S} f(x)\\
    &\ge p|V(G)-T-S|+(p-2)|T_1|+(p-4)|T_2|+2|T_1|+4|T_2|+3= p|V(G)-S|+3,
\end{align*}
a contradiction.

\section{Knitted subgraph in dense graphs: a proof of Lemma~\ref{knitted1}}\label{finding}

In this section, we prove Lemma~\ref{knitted1}.
We will only give the detailed proof of the case where $p=42$.  The proofs for $p=30$ and for $p=18$ are similar but more tedious, 
so these proofs will be relegated to a pair of appendices.

Whether we have $n(H) \leq k$ and $\delta(H) \geq \lfloor \frac{k}{2} \rfloor$ or $n(H) \leq k - 2$ and $\delta(H) = \lfloor \frac{k}{2} \rfloor$, we have $n(H) \leq 2\delta(H)- 1$, so it suffices to prove the following lemma:

\begin{lemma}
\label{l:23case}
Let $H$ be a graph, $v \in V(H)$ such that $H = N[v]$.
Suppose $\delta(H) \geq 21$ and $|H| = n \leq \min \{ 2 \delta(H) - 1, 42 \}$.
Then $H$ has a $(2,2,2,2,1)$-knitted subgraph.
\end{lemma}

Before proving this lemma, we will introduce the notation we will make use of in this proof as well as in the proofs in the appendices.
If a graph $H$ were a counterexample to Lemma~\ref{l:23case}, then $H$ itself would not be $(2,2,2,2,1)$-knitted, so there would be vertices $u_0, u_1, v_1, u_2, v_2, u_3, v_3, u_4, v_4 \in V(H)$ such that $H$ would not have $(u_1, v_1)$-, $(u_2, v_2)$-, $(u_3, v_3)$-, and $(u_4, v_4)$-paths that would be disjoint from each other and from $u_0$.
We define $C = C_0 \cup C_1 \cup C_2 \cup C_3 \cup C_4 \subseteq V(H)$ as follows:
\begin{enumerate}[(i)]
	\item $C_0 = \{ u_0 \}$.
	\item For $i \in \{ 1, 2, 3 \}$, if $C_0, \dotsc, C_{i-1}$ have been defined and the graph $H \setminus \left( \bigcup_{j=0}^{i-1} C_j \cup \{ u_{j+1}, v_{j+1}, \dotsc, u_4, v_4 \} \right)$ has a $(u_i, v_i)$-path with at most 5 vertices, then $C_i$ is the vertex set of that path. Otherwise, $C_i = \{ u_i, v_i \}$. (In this latter case, $C_i$ is necessarily disconnected.)
	\item $C_4 = \{ u_4, v_4 \}$.
	\item Subject to (i)-(iii), rearranging the pairs $(u_1, v_1), \dotsc, (u_4, v_4)$ if necessary, as many of the $C_i$ as possible induce connected subgraphs of $H$.
	\item Subject to (i)-(iv), rearranging the pairs $(u_1, v_1), \dotsc, (u_4, v_4)$ if necessary, $C$ has as few vertices as possible.
\end{enumerate}
In particular, for every $i \in [4]$, either $C_i$ is the vertex set of an induced $(u_i, v_i)$-path or $C_i = \{ u_i, v_i \}$.
Rearranging if necessary, we may assume $C_1, \dotsc, C_s$ induce connected paths and $C_{s+1}, \dotsc, C_4$ are pairs of non-adjacent vertices.
We may also assume $|C_1| \leq \dotsb \leq |C_s|$.

Suppose $z_1, y, z_2$ are three consecutive vertices on some $C_i$, and let $C_j = \{ u_j, v_j \}$ be a pair of non-adjacent vertices.
If there is a $(u_j, v_j)$-path in $H$ whose only internal vertex in $C$ is $y$, and if there is a vertex $x \in H - C$ that is adjacent to both $z_1$ and $z_2$, then our original choice of $C$ was not minimal with respect to the number of components: we can replace the segment $z_1 y z_2$ on $C_i$ with $z_1 x z_2$, and we can replace $C_j$ with that $(u_j, v_j)$-path that goes through $y$.
We will refer to this operation as an \textbf{$(x, y)$-reroute} of $C_i$ and $C_j$. For each $i\in\{0,1,2,3,4\}$, we call a vertex $u$ is complete to a vertex set $U$ if $u$ is adjacent to every vertex in $U$,
 $u$ is anticomplete to $U$ if $u$ is adjacent to no vertex in $U$.

 The following two lemmas, which may be of independent interest, provide powerful tools in our proofs. 
 
\begin{lemma}
\label{l:si}
For $j \in [4]$ such that either $H[C_j]$ is disconnected or $|C_j| \geq 6$, let $A_j = N(u_j) - C$ and $B_j = N(v_j) - C$.
In the graph $H - (C - \{ u_j, v_j \})$, let $A$ be the component containing $u_j$ and let $B$ be the component containing $v_j$.
Let $(A^*, B^*)$ be either the pair $(A, B)$ (in the case where $H[C_j]$ is disconnected) or the pair $(A_j, B_j)$ (in either the case where $H[C_j]$ is disconnected or the case where $|C_j| \geq 6$).
Let $a \in A^* - u_j$ and $b \in B^* - v_j$, and, for $i\ne j$ and $i\in \{ 0, 1, \dotsc, 4 \}$, let \[ s_i = |N(a) \cap N(b) \cap C_i| - |C_i - (N(a) \cup N(b))|. \]
\begin{enumerate}[(a)]
	\item If $H[C_i]$ is disconnected, then $s_i \leq 0$. 
	\item If $H[C_i]$ is connected, then no two neighbors of $a$ on $C_i$ have at least two vertices between them on the path $H[C_i]$. In particular, each of $a$ and $b$ has at most 3 neighbors in $C_i$. 
	\item If $H[C_i]$ is connected, then $-|C_i| \leq s_i \leq \min \{ |C_i|, 6 - |C_i| \}$. 
	\item Let $t_a = |A^* - N[a]|$ and $t_b = |B^* - N[b]|$. Then $\sum_{i=0}^4 s_i \geq d(a) + d(B) - (|H| - 2) + t_a + t_b$. 
	\item If $(A^*, B^*) = (A, B)$, $H[A]$ and $H[B]$ are 2-connected, $a \neq u_j$, $b \neq v_j$, and $|H - (A \cup B)| \leq \delta(H) - 2$, then $s_i \neq 3$. 
\end{enumerate}
\end{lemma}
\begin{proof}
\begin{enumerate}[(a)]
	\item 	If $H[C_i]$ is disconnected for some $i \neq j$, then $C_i = \{ u_i, v_i \}$.
	If either $a$ or $b$ were complete to $C_i$, then we could add that vertex to $C_i$ to make $H[C_i]$ connected, contrary to the definition of $C$, so it must be the case that each of $a$ and $b$ has at most 1 neighbor in $C_i$.
	If $a$ and $b$ have a common neighbor in $C_i$, say $u_i$, then neither one is adjacent to $v_i$, so that $s_i = 0$.
	If $a$ and $b$ have no common neighbor in $C_i$, then $s_i \leq |N(a) \cap N(b) \cap C_i| = 0$.
	\item If $x_1 x_2 \dotsc x_k$ are consecutive vertices on $C_i$ such that $x_1, x_k \in N(a)$, then we can replace the segment $x_1 x_2 \dotsc x_k$ of $C_i$ with $x_1 a x_k$ to get a different choice for $C_i$.
	Since $|C_i|$ was chosen to be minimal, this different choice for $C_i$ cannot have fewer vertices than the original choice for $C_i$, so we much have $k \leq 3$, that is, the two neighbors of $a$ cannot have more than 1 vertex between them.
	By symmetry, the same is true for $b$.
	\item Clearly $s_i \geq -|C_i - (N[a] \cup N[b])| \geq -|C_i|$ and $s_i \leq |C_i \cap N(a) \cap N(b)| \leq |C_i|$.
	In the case where $a \neq u_j$ and $b \neq v_j$, suppose $a$ and $b$ have $t$ common neighbors on $C_i$.
	By part (b), each of $a$ and $b$ has at most 3 neighbors on $C_i$, so $a$ is adjacent to at most $3-t$ vertices of $C_i$ that are not neighbors of $b$ and vice versa.
	We then have $|[N(a) \cup N(b)] \cap C_i| \leq t + (3-t) + (3-t) = 6-t$.
	Thus \[ s_i = |N(a) \cap N(b) \cap C_i| - |C_i| + |[N(a) \cup N(b)] \cap C_i| \leq t - |C_i| + 6-t = 6-|C_i|. \]
	\item We have
	\begin{align*}
	|N(a) \cap N(b) \cap C| &= |N(a) \cap N(b)| \\
	&= |N(a)| + |N(b)| - |N(a) \cup N(b)| \\
	&= d(a) + d(b) - |H| + |H - [N(a) \cup N(b)]| \\
	&\geq d(a) + d(b) - |H| + |\{ a, b \}| + |A - N[a]| + |B - N[b]| + |C - (N[a] \cup N[b])| \\
	&\geq d(a) + d(b) + (|H| - 2) + t_a + t_b + |C - (N[a] \cup N[b])|
	\end{align*}
	(where $t_a = |A - N[a]|$ if $A^* = A$ and $t_a \leq |A - N[a]|$ if $A^* = A_j$).
	It follows that \[ \sum_{i=0}^4 s_i = |N(a) \cap N(b) \cap C| - |C - (N[a] \cup N[b])| \geq d(a) + d(b) + (|H| - 2) + t_a + t_b. \]
	\item If $|C_i| \neq 3$, then $s_i \leq 2$ by part (c), so we may assume $|C_i| = 3$; let $x$ be its middle vertex.
	If $s_i = 3$, then $\{ a, b \}$ is complete to $C_i$.
	This implies that $N(x) \cap (A \cup B) = \{ a, b \}$: otherwise, if there is $y \in N(x) \cap (A \cup B)$ that is neither $a$ nor $b$, we can perform an $(x, a)$- or $(x, b)$-reroute of $C_i$ and $C_j$.	Then, because $|N(x) \cap (A \cup B)| = 2$, we must have \[ \delta(H) \leq d(x) = |N(x) \cap (A \cup B)| + |N(x) - (A \cup B)| \leq 2 + |H - (A \cup B \cup \{ x \})|, \]
	implying that $|H - (A \cup B \cup \{ x \})| \geq \delta(H) - 2$ and so $|H - (A \cup B)| \geq \delta(H) - 1$.
\end{enumerate}\end{proof}

\begin{lemma}
\label{l:si2} 
Define $A^*$ and $B^*$ as in the previous lemma.
Let $a, a' \in A^* - u_j$ and $b, b' \in B^* - v_j$ be four distinct vertices, and, for $i \in \{ 0, 1, \dotsc, 4 \}$, let \[ s_i = |N(a) \cap N(b) \cap C_i| - |C_i - (N[a] \cup N[b])| \text{ and } s_i' = |N(a') \cap N(b') \cap C_i| - |C_i - (N[a'] \cup N[b'])|. \]
Suppose either $(A^*, B^*) = (A_j, B_j)$ or $H[A^*]$ and $H[B^*]$ are 2-connected, and suppose $s_i + s_i' \geq 3$ and $s_i\ge s_i'$.
\begin{enumerate}[(a)]
	\item $s_i + s_i' \in \{ 3, 4 \}$ and $|C_i| \in \{ 2, 3 \}$.
	\item If $|C_i| = 3$ and some vertex in $A$ is complete to $C_i$, then the middle vertex of $C_i$ is anticomplete to $B$.
	\item If $s_i + s_i' = 4$, then $s_i = s_i' = 2$ and each vertex in $\{ a, a', b, b' \}$ is complete to $\{ u_i, v_i \}$.
\end{enumerate}
\end{lemma}
\begin{proof}
\begin{enumerate}[(a)]
	\item By Lemma~\ref{l:si}(c), we have $s_i \leq 3$ and $s_i' \leq 3$ for each $i$.
	
	Suppose $s_i = 3$. 	Then $|C_i| = 3$, and $a$ and $b$ are complete to $C_i$.
	If we call the middle vertex of $C_i$ $x$, then $x$ is anticomplete to $\{ a', b' \}$, otherwise we can perform an $(x, a)$- or $(x, b)$-reroute of $C_i$ and $C_j$. 
	Moreover, neither $a'$ nor $b'$ can be complete to $\{ u_i, v_i \}$, otherwise we can perform an $(x, a')$- or $(x, b')$-reroute of $C_i$ and $C_j$. 
	So each of $a'$ and $b'$ has at most 1 neighbor in $C_i$; whether they have 1 common neighbor or 0, we get $s_i' \leq -1$ and so $s_i + s_i' \leq 2$.
	Thus, if $s_i + s_i' \geq 3$, it must be the case that neither $s_i$ nor $s_i'$ is equal to 3.
	We must then have $\max \{ s_i, s_i' \} = 2$, so that $s_i + s_i' \leq 4$, with equality if and only if $s_i = s_i' = 2$. Note that $s_i\ge s_i'$. So $s_i = 2$ and $s_i' \in \{ 1, 2 \}$.
	By Lemma~\ref{l:si}(c), we then have $2 \leq |C_i| \leq 4$.
	
	Suppose $|C_i| = 4$.
	Since $s_i = 2$, we either have $|N(a) \cap N(b) \cap C_i| = 2$ and $|C_i - (N[a] \cup N[b])| = 0$ or $|N(a) \cap N(b) \cap C_i| = 3$ and $|C_i - (N[a] \cup N[b]) = 1$.
	In either case, if we label $C_i$ as $u_i x y v_i$, then $\{ a, b \}$ is complete to $\{ x, y \}$, and each of $a$ and $b$ is adjacent to either $u_i$ or $v_i$.
	But then, if we assume $a$ is adjacent to $u_i$, then $x$ cannot be adjacent to $a'$, otherwise we can perform an $(x, a)$-reroute of $C_i$ and $C_j$.
	Similarly, there is an internal vertex of $C_i$ that is not adjacent to $b'$ (either $x$ or $y$, depending on whether $b$ is adjacent to $u_i$ or $v_i$).
	This means that neither $a'$ nor $b'$ can have 3 consecutive neighbors on $C_i$, which, by Lemma~\ref{l:si}(b), implies that neither $a'$ nor $b'$ can have 3 neighbors on $C_i$ at all.
	But if $\max \{ |N(a') \cap C_i|, |N(b') \cap C_i| \} \leq 2$, it must be the case that $s_i' \leq 0$, so $s_i + s_i' < 3$.
	So, if $s_i + s_i' \geq 3$, we must have $|C_i| \neq 4$.
	\item Suppose $|C_i| = 3$, and
	let $a^* \in A$ be complete to $C_i$; note that we do not require that $a^* \neq u_j$.
	Since $s_i=2$, it must be the case that, of $a$ and $b$, one is complete to $C_i$ and the other has exactly two neighbors in $C_i$.
	Suppose $b$ is complete to $C_i$.
	Then no other vertex of $B$ (not even $v_j$) is adjacent to the middle vertex $x$ of $C_i$, otherwise we could perform an $(x, b)$-reroute of $C_i$ and $C_j$ through $a^*$. 
	In particular, $b'$ is not adjacent to $x$. Since $s_i'\in\{1,2\}$, $b'$ is adjacent to at least one of $u_i$ and $v_i$. If $b'$ is adjacent to both $u_i$ and $v_i$, then we could perform an $(x, b')$-reroute of $C_i$ and $C_j$ through $a^*$. So by symmetry we may assume that $N(b')\cap C_i=\{u_i\}$. Then
 $a'$ must be complete to $C_i$ since $b'\ge1$. Now we could perform an $(x, a')$-reroute or $(x, a)$-reroute of $C_i$ and $C_j$, a contradiction. 
	Thus, if $|C_i| = 3$, then $b$ is not complete to $C_i$, so $a$ is complete to $C_i$ and the neighbors of $b$ on $C_i$ must be $u_i$ and $v_i$, for otherwise we can perform an  $(x, a')$-reroute or  $(x, b')$-reroute of $C_i$ and $C_j$ through $a$.
	Then $x$ is not adjacent to any vertex of $B$, otherwise we could perform an $(x, b)$-reroute of $C_i$ and $C_j$.
    
	\item Clearly, if $|C_i| = 2$ and $s_i + s_i' = 4$, then $\{ a, a', b, b' \}$ must be complete to $\{ u_i, v_i \}$.
	If $|C_i| = 3$, then, since $s_i = 2$, either $a$ or $b$ (without loss of generality, $a$) is complete to $C_i$, with the other vertex (in this case, $b$) having exactly 2 neighbors in $C_i$.
	Then, by part (b), the middle vertex $x$ of $C_i$ has no neighbor in $B$, so the neighbors of $b$ on $C_i$ must be $u_i$ and $v_i$ exactly.
	Likewise, since $s_i' = 2$ and $b'$ is not adjacent to $x$, $a'$ is complete to $C_i$ and $N(b') \cap C_i = \{ u_i, v_i \}$, so that each vertex in $\{ a, a', b, b' \}$ is complete to $\{ u_i, v_i \}$, as desired.
\end{enumerate}
\end{proof}

We will make use of the following results to show when a graph is $(2,2,2,2,1)$-knitted in this proof and when a graph is 4-linked or $(2,2,2,1)$-knitted in the proofs in the appendices:

\begin{proposition}
\label{p:prop1}
Let $k \geq 3$ be an integer, and let $L$ be a graph with $|L| \geq 2k+1$.
Suppose that every set of $k$ pairs of vertices in $L$ can be labeled $(u_1, v_1), \dotsc, (u_k, v_k)$ in such a way that, for every $i \in [k]$, either $u_i v_i \in E(L)$ or $u_i$ and $v_i$ have at least $2k-2+i$ common neighbors.
Then $L$ is $k$-linked.
Moreover, if each of these non-adjacent pairs has at least $2k-1+i$ common neighbors, then $L$ is $(2, \dotsc, 2, 1)$-knitted (with $k$ 2s).
\end{proposition}
\begin{proof}
Let $u_0, u_1, \dotsc, u_k, v_1, \dotsc, v_k \in V(L)$ be distinct.
If $u_1 v_1 \in E(L)$, we can connect $u_1$ to $v_1$ with a path of 2 vertices; if $u_1 v_1 \notin E(L)$ and $u_1$ and $v_1$ have at least $2k-2+1 = 2(k-1) + 1$ common neighbors, then they have a common neighbor that is not in $\{ u_2, \dotsc, u_k, v_2, \dotsc, v_k \}$, so we can use this common neighbor to connect them with a path on at most 3 vertices; if $u_1$ and $v_1$ have at least $2k-1+1 = 2(k-1) + 1 + 1$ common neighbors, then we can choose this common neighbor to be distinct from $u_0$ as well.
Suppose that, for some $i \geq 2$, we have connected the pairs $(u_1, v_1), \dotsc, (u_{i-1}, v_{i-1})$ with paths on at most 3 vertices.
If $u_i v_i \in E(L)$, we can connect $u_i$ to $v_i$ with a path of 2 vertices; if $u_i v_i \notin E(L)$ and $u_i$ and $v_i$ have at least $2k-2+i = 2(k-1) + (i-1) + 1$ common neighbors, then they have a common neighbor that is not in $\{ u_1, \dotsc, u_{i-1}, u_{i+1}, \dotsc, u_k, v_1, \dotsc, v_{i-1}, v_{i+1}, \dotsc, v_k \}$, and is not one of the interior vertices of the paths connecting the pairs $(u_1, v_1), \dotsc, (u_{i-1}, v_{i-1})$, so we can use this common neighbor to connect $u_i$ to $v_i$ with a path on at most 3 vertices. Moreover, if $u_i$ and $v_i$ have at least $2k-1+i = 2(k-1) + (i-1) + 1 + 1$ common neighbors, then we can choose this common neighbor to be distinct from $u_0$ as well.
\end{proof}

\begin{corollary}
\label{c:common}
Let $k \geq 3$ be an integer, and let $L$ be a graph with $|L| \geq 2k+1$.
If every pair of non-adjacent vertices in $L$ has at least $3k-2$ common neighbors, then $L$ is $k$-linked, and if every pair of non-adjacent vertices in $L$ has at least $3k-1$ common neighbors, then $L$ is $(2, \dotsc, 2, 1)$-knitted (with $k$ 2s).
\end{corollary}

\begin{corollary}
\label{c:uncommon}
Let $k \geq 3$ be an integer, and let $L$ be a graph with $|L| \geq 2k+1$.
Suppose there is a vertex $v \in V(L)$ such that, for every $x \in V(L) - N[v]$, $v$ and $x$ have at least $2k-1$ (respectively $2k$) common neighbors, and for every non-adjacent pair $x, y \in V(L) - v$, $v$ and $x$ have at least $3k-2$ (respectively $3k-1$) common neighbors.
Then $L$ is $k$-linked (respectively $(2, \dotsc, 2, 1)$-knitted with $k$ 2s).
\end{corollary}

We are now ready to finish a proof for Lemma~\ref{l:23case}.

\begin{proof}[Proof of Lemma~\ref{l:23case}]
Suppose $H$ has no $(2,2,2,2,1)$-knitted subgraph.
This implies that $H$ has no 9-clique, and, by Proposition~\ref{p:prop1}, no subgraph $L$ such that every non-adjacent pair of vertices in $L$ has at least 11 common neighbors in $L$.
Moreover, $H$ itself is not $(2,2,2,2,1)$-knitted, so there are vertices $u_0, u_1, v_1, u_2, v_2, u_3, v_3, u_4, v_4 \in V(H)$ such that $H$ does not have $(u_1, v_1)$-, $(u_2, v_2)$-, $(u_3, v_3)$-, and $(u_4, v_4)$-paths that are disjoint from each other and from $u_0$.

\begin{claim}
$s \geq 3$ for all $i \in [3]$.
\end{claim}
\begin{proof}
If not, pick $r \in [3]$ such that $r-1$ is the largest index for which $H[C_{r-1}]$ is connected; by definition of $C$, we then have $|C_j| \leq 5$ for $j \leq r-1$ and $|C_j| = 2$ for $j \geq r$.
Let $A_r = N(u_r) - C$ and $B_r = N(v_r) - C$.
We may assume $d_{H - C} (A_r, B_r) > 2$.
We have
	\begin{align*}
	|A_r| &= d(u_r) - |N(u_r) \cap C| \\
	&\geq d(u_r) - |C_0| - \sum_{j=1}^{r-1} |C_j| - |N(u_r) \cap C_r| - \sum_{j=r+1}^4 |C_j| \\
	&\geq \delta(H) - 1 - 5(r-1) - 1 -2(4-r) \\
	&= \delta(H) - 3r -5 \\
	&\geq \delta(H) - 14 \geq 7.
	\end{align*}
Recall that $V(H)=N[v]$. Since no vertex in $A_r \cup \{ u_r \}$ has a neighbor in $B_r$, $v \notin A_r$, so $v$ is complete to $A_r \cup \{ u_r \}$.
Since $H$ has no 9-clique, this implies that $A_r$ cannot be a 7-clique, so $\Delta(\overline{H}[A]) \geq 1$.
Likewise, $\Delta(\overline{H}[B]) \geq 1$.
Choose $a, a' \in A_r$ and $b, b' \in B_r$ such that $aa', bb' \notin E(H)$.
Defining $s_i$ and $s_i'$ as in Lemma~\ref{l:si}, by Lemma~\ref{l:si}(d) we have \[ \sum_{i=0}^4 s_i = |N(a) \cap N(b) \cap C| - |C - (N[a] \cup N[b])| \geq 3 + 1 + 1 = 5 \] and likewise $\sum_{i=0}^4 s_i' \geq5$, thus \[ \sum_{i=0}^4 (s_i + s_i') \geq 10. \]
By Lemma~\ref{l:si}(a), we have $s_4 + s_4' \leq 0$, and by definition of $s_i$ and $s_i'$, we must have $s_0 + s_0' \leq 2$.
Because $C_3$ is disconnected, we have $s_3 + s_3' \leq 0$ by Lemma~\ref{l:si}(a). 
Thus, we have \[ \sum_{i=1}^2 (s_i + s_i') \geq 8. \]
By Lemma~\ref{l:si2}(a), we must then have $s_1 = s_1' = s_2 = s_2' = 2$ and $|C_1|, |C_2| \in \{ 2, 3 \}$, so that $r = 3$.

Since we have equality here, we must have $\Delta(\overline{H}[A_3]) = \Delta(\overline{H}[B_3]) = 1$.
It follows that
	\begin{align*}
	|A_3| &\geq d(u_r) - |C_0| - |C_1| - |C_2| - |N(u_r) \cap C_3| - |C_4| \\
	&\geq \delta(H) - (1 + 3 + 3 + 1 + 2) \\
	&\geq 21 - 10 = 11.
	\end{align*}
Then $|A_3 \cup \{ u_3 \}| \geq 12$, so every pair of non-adjacent vertices in $A_3 \cup \{ u_3 \}$ has at least 10 common neighbors in $A_3 \cup \{ u_3 \}$, hence every pair of non-adjacent vertices in $A_3 \cup \{ u_3, v \}$ has at least 11 common neighbors in $A_3 \cup \{ u_3, v \}$.
Thus, by Proposition~\ref{p:prop1}, $A_3$ is $(2,2,2,2,1)$-knitted.
\end{proof}

\begin{claim}
$t=4$.
\end{claim}
\begin{proof}
Let $A_4 = N(u_4) - C$ and $B_4 = N(v_4) - C$.
In $H - (C - \{ u_4, v_4 \})$, let $A$ be the component containing $u_4$ and let $B$ be the component containing $v_4$, so that $A_4 \subseteq V(A)$ and $B_4 \subseteq V(B)$; if there is no $(u_4, v_4)$-path in $H - (C - \{ u_4, v_4 \})$, then it must be the case that $A \neq B$.
By Claim 5.7, we have \[ |A_4| \geq d(u_4) \geq \delta(H) - (1 + |C_1| + |C_2| + |C_3| + 1) \geq 21 - 17 = 4. \]
For any $a \in A_4 - u_4$, we have $N[a] \subseteq A \cup C$, so \[ |A| \geq d(a) + 1 - |N(a) \cap (C - A)| \geq \delta(H) + 1 - (1 + 3 + 3 + 3 + 0) = \delta(H) - 9 \geq 12. \]
If $\Delta(\overline{H}[A - u_4]) \leq 1$, then every pair of nonadjacent vertices in $A$ would have at least 11 common neighbors in $A \cup \{ v \}$, so $H[A]$ would be $(2,2,2,2,1)$-knitted, contrary to our choice of $H$.
Thus $\Delta(\overline{H}[A - u_4]) \geq 2$ and likewise $\Delta(\overline{H}[B - v_4]) \geq 2$.
Let $a \in A$ and $b \in B$ have maximum degree in $\overline{H}[A - u_4]$ and $\overline{H}[B - v_4]$, respectively; let $a' \in A - \{ a, u_4 \}$ have at least 1 non-neighbor in $A - u_4$, and let $b' \in B - \{ b, v_4 \}$ have at least 1 non-neighbor in $B - v_4$. Denote $a_1,a_2$ be two non-neighbors of $a$ in $A$. Then we have \[ |A| -| \{a,a_1,a_2\}| \geq d(a) - |N(a) \cap (C - A)| \geq \delta(H) - 10, \] so that $|A| \geq \delta(H) - 7 \geq 14$ and likewise $|B| \geq 14$.
Then, by Lemma~\ref{l:si}(d), we have $\sum_{i=0}^4 s_i \geq 7$ and $\sum_{i=0}^4 s_i' \geq 5$.
Conversely, since we must have \[ 5 + \Delta(\overline{H}[A - u_4]) \leq 3 + \Delta(\overline{H}[A - u_4]) + \Delta(\overline{H}[B - v_4]) \leq \sum_{i=0}^4 s_i \leq 1 + 3 + 3 + 3 + 0 = 10, \] we have $\Delta(\overline{H}[A - u_4]) \leq 5$; this implies that every pair of vertices in $A - u_4$ has at least $|A - u_4| - 12 \geq 3$ common neighbors, so $A$ must be 2-connected. 
We have \[ |H - (A \cup B)| = |H| - |A| - |B| \leq 2 \delta(H) - 1 - \big( \delta(H) - 9 \big) - \big( \delta(H) - 9 \big) = 17 \leq \delta(H) - 2, \] so, by Lemma~\ref{l:si}(e), we have $s_i \leq 2$ for $i \in [3]$. Since $s_0 \leq 1$ and $s_4 \leq 0$, we must have $\sum_{i=0}^4 s_i = 7$ exactly, which, in turn, implies that $\Delta(\overline{H}[A]) = \Delta(\overline{H}[B]) = 2$.
But then every pair of non-adjacent vertices in $A$ has at least $|A| - 4 \geq 11$ common neighbors in $A$, so $H[A]$ is $(2,2,2,2,1)$-knitted, a contradiction.
\end{proof}
\end{proof}

\section{Concluding Remarks}\label{remark}

In~\cite{M68}, Mader utilized a result on rooted $K_4$ minors after Theorem~\ref{thm:MaderS-3} to finish the proof of Theorem~\ref{thm:7Conn}.
It is an open question whether an analogous result can be found for rooted $K_5$ minors.

\begin{question}\label{Q:RootedK5}
Let $\{v_1, \dots, v_5\} \subseteq V(G)$ such that $\alpha(\{v_1, \dots, v_5\}) = 2$.
For $i \in \{1, \dots, 5\}$, let $V_i \subseteq V(G)$ be disjoint subsets with $v_i \in V_i$.
Assume for all $i \ne j$ that there exists a $v_i, v_j$-path consisting only of vertices from $V_i \cup V_j$.
Do there exist disjoint subsets $V_1', \dots, V_5' \subseteq V(G)$ such that $v_i \in V_i'$, $G[V_i']$ is connected, and there exists at least one $V_i', V_j'$-edge for all $i \ne j$?
\end{question}

If Question~\ref{Q:RootedK5} can be answered in the affirmative, then Theorem~\ref{main2} could be improved to say that any $k$-contraction-critical graph is $8$-connected for $k \ge 11$.
We would obtain $k \ge 11$ here, as opposed to a more desirable $k \ge 8$, since Theorem~\ref{main} requires a $(k + 4)$-contraction-critical graph.
This represents three additional colors when compared to Theorem~\ref{thm:MaderS-3}, and these colors are reflected in the bound.
Answering Question~\ref{Q:RootedK5}, however, seems hard.

\newpage

\section{Appendix}

\begin{lemma}
Let $H$ be a graph, $v \in V(H)$ such that $H = N[v]$.
Suppose $\delta(H) \geq 15$ and $|H| = n \leq \min \{ 2 \delta(H) - 1, 30 \}$.
Then $H$ has a $4$-linked subgraph.
\end{lemma}
\begin{proof}
Suppose $H$ has no 4-linked subgraph.
This implies that $H$ has no 8-clique, and, by Proposition~\ref{p:prop1}, no subgraph $L$ such that $|L| \geq 8$ and every non-adjacent pair of vertices in $L$ has at least 10 common neighbors in $L$.
Since $H$ itself is not 4-linked, we will define $u_1, v_1, \dotsc, u_4, v_4$ as in the proof of Lemma~\ref{l:23case}, and define $C = C_1 \cup C_2 \cup C_3 \cup C_4 \subseteq V(H)$ as follows:
\begin{enumerate}[(i)]
    \item If the graph $H \setminus \{ u_1, v_1, u_2, v_2, u_3, v_3, u_4, v_4 \}$ has a $(u_1, v_1)$-path on at most 3 vertices, then $C_1$ is the vertex set of this path. Otherwise, $C_1 = \{ u_1, v_1 \}$.
    \item If the graph $H \setminus (C_1 \cup \{ u_2, v_2, u_3, v_3, u_4, v_4 \})$ has a $(u_2, v_2)$-path on at most 5 vertices, then $C_2$ is the vertex set of this path. Otherwise, $C_2 = \{ u_2, v_2 \}$.
    \item If the graph $H \setminus (C_1 \cup C_2 \cup \{ u_3, v_3, u_4, v_4 \})$ has a $(u_3, v_3)$-path on at most 7 vertices, then $C_3$ is the vertex set of this path. Otherwise, $C_3 = \{ u_3, v_3 \}$.
	\item $C_4 = \{ u_4, v_4 \}$.
	\item Subject to (i)-(iv), rearranging the pairs $(u_1, v_1), \dotsc, (u_4, v_4)$ if necessary, as many of the $C_i$ as possible induce connected subgraphs of $H$.
	\item Subject to (i)-(v), rearranging the pairs $(u_1, v_1), \dotsc, (u_4, v_4)$ if necessary, $C$ has as few vertices as possible.
\end{enumerate}
We may again assume, rearranging if necessary, that there is $t \in \{ 0, 1, 2, 3, 4 \}$ such that $H[C_i]$ is connected for all $i \leq t$ and that $|C_i| \leq |C_j|$ whenever $i < j \leq t$.
Lemma~\ref{l:si} still holds in this case if we omit $s_0$.
Since no $u_i$ and no $v_i$ is complete to every other vertex in $C$ (since $u_i v_i \notin E(C)$ for each $i$), we have $u_1 v v_1$ as an option for $C_1$, so that $t \geq 1$ and $|C_1| \leq 3$.

\begin{claim}
$t \geq 2$. 
\end{claim}
\begin{proof}
If not, we define $A_2$ as in Lemma~\ref{l:si} and get \[ |A_2| \geq d(u_2) - |C_1| - 1 - |C_3| - |C_4| \geq \delta(H) - 8 \geq 7. \]
Since $u_2$ is complete to $A_2$ and $H$ has no 8-clique, $A_2$ cannot be a 7-clique, so $\Delta(\overline{H}[A_2]) \geq 1$; likewise, $\Delta(\overline{H}[B_2]) \geq 1$. 
Then, if we take $a \in A_2$ with a non-neighbor in $A_2$ and $b \in B_2$ with a non-neighbor in $B_2$, by Lemma~\ref{l:si}(d), we have \[ \sum_{i=1}^4 s_i \geq 5. \]
Part (a) of this lemma gives us $s_3 + s_4 \leq 0$.
Since $C_2$ is disconnected, we have $s_2 \leq 0$, but then $s_1 \geq 5$, contrary to Lemma~\ref{l:si}(c).
\end{proof}

\begin{claim}
$t \geq 3$. 
\end{claim}
\begin{proof}
If not, we define $A_3$ in Lemma~\ref{l:si} and get \[ |A_3| \geq d(u_3) - (|C_1| + |C_2| + 1 + |C_4|) \geq \delta(H) - (3 + 5 + 1 + 2) \geq 4. \]
Then $A_3$ is nonempty; let $A_3' = \{ x \in V(H) - (A_3 \cup C) : N(x) \cap A_3 \neq \varnothing \}$.
We define $B_3$ and $B_3'$ similarly; note that, in $H - C$, the distance from $A_3 \cup A_3'$ to $B_3 \cup B_3'$ is at least 2, otherwise we get a $(u_3, v_3)$-path of length at most 7.
For any $a \in A_3$, \[ |A_3 \cup A_3'| \geq d(a) + 1 - |N(a) \cap [C - (A_3 \cup A_3')]| \geq \delta(H) + 1 - (3 + 3 + 1 + 1) = \delta(H) - 7 \geq 8. \]
Since $H$ has no 8-clique, $\Delta(\overline{H}[A_3 \cup A_3']) \geq 1$, and likewise $\Delta(\overline{H}[B_3 \cup B_3']) \geq 1$.
Taking $a \in A_3 \cup A_3'$ that is not complete to $(A_3 \cup A_3') - a$ and $b \in B_3 \cup B_3'$ that is not complete to $(B_3 \cup B_3') - b$, we have \[ \sum_{i=1}^4 s_i \geq 5. \]
Since $C_4$ is disconnected, $s_4 \leq 0$ by part (a) or (c) of Lemma~\ref{l:si}, and, since $C_3$ is disconnected, 
$s_3 \leq 0$ as well.
Then $s_1 + s_2 \geq 5$, so either $s_1$ or $s_2$ is at least 3; without loss of generality, $s_1 \geq 3$.
By Lemma~\ref{l:si}(c), $s_1 = 3$ and $|C_1| = 3$.
We claim that the middle vertex of $C_1$, say $x$, can have no neighbor in $(A_3 \cup A_3' \cup B_3 \cup B_3') - \{ a, b \}$.
Otherwise, suppose it has such a neighbor, say $a' \in A_3'$.
Note that there must be a path from $a'$ to $u_3$ in $(A_3 \cup A_3') - a$; if $a' \in A_3$, this is immediate.
If $a' \in A_3'$ and its only neighbor in $A_3$ is $a$, then there are 3 vertices in $A_3 \cup A_3'$ that are not neighbors of $a'$, so, applying Lemma~\ref{l:si}(d) to $a'$ and $b$, we get \[ \sum_{i=1}^4 s_i \geq 7, \] contrary to the fact that $\max\{ s_1, s_2 \} \leq 3$ and $\max \{ s_3, s_4 \} \leq 0$.
We could then replace $C_1$ with $u_1 a v_1$ and replace $C_3$ with a path of length at most 2 from $u_3$ to $a'$, then $x$, then a path of length at most 2 from $b$ to $v_3$ in $B$, contrary to the choice of $C$.
Thus $x$ has at most 1 neighbor in each of $A_3 \cup A_3'$ and $B_3 \cup B_3'$, so we get \[ d(x) \leq |H| - (|A_3 \cup A_3'| - 1) - (|B_3 \cup B_3'| - 1) - |\{ x \}| \leq \big( 2\delta(H) - 1 \big) - 2 \big(\delta(H) - 8 \big) - 1 = 14 < \delta(H), \] a contradiction. 
\end{proof}

Now we may assume $t = 3$ and define $A_4$ and $A$ as in Lemma~\ref{l:si}.

\begin{claim}
$A_4 \neq \varnothing$ and $B_4 \neq \varnothing$.
\end{claim}
\begin{proof}
Assume without loss of generality that $A_4 = \varnothing$.
We have \[ |A_4| \geq d(u_4) - |C - \{ u_4, v_4 \}| \geq \delta(H) - (3 + 5 + 7) \geq 0, \] so, if $A_4$ is empty (in which case $A = \{ u_3 \}$), we must have $|C_1| = 3$, $|C_2| = 5$, $|C_3| = 7$, and $N(u_4) = C_1 \cup C_2 \cup C_3$.
In that case, $v_4$ is not adjacent to any internal vertex of $C_2$ or $C_3$, as that vertex would then be a common neighbor of $u_4$ and $v_4$; if, say, $x \in N(u_4) \cap N(v_4)$ was an internal vertex of $C_2$, we could replace $C_4$ with $u_4 x v_4$ to get a path shorter than $C_2$, contrary to the minimality of $C$.
Then \[ |B_4| \geq d(v_4) - |N(v_4) \cap C| \geq \delta(H) - (3 + 2 + 2 + 0) \geq 8. \]
Let $B_4'$ be the set of vertices in $H - (C \cup B_4)$ that have a neighbor in $B_4$, so that every vertex in $B_4$ has all of its neighbors in $C \cup B_4 \cup B_4'$.
Note that no vertex in $B_4$ can be adjacent to an interior vertex of $C_2$ or $C_3$, as that would give us a $(u_4, v_4)$-path on at most 4 vertices, a contradiction; moreover, no vertex in $B_4$ is adjacent to both ends of $C_2$ or both ends of $C_3$, as that would allow us to replace $C_2$ or $C_3$ with a path on 3 vertices.
We then have, for any $b \in B_4$, \[ |B_4 \cup B_4'| \geq |N[b] - C| \geq \delta(H) + 1 - (3 + 1 + 1 + 1) = \delta(H) - 5 \geq 10. \]
Now consider an interior vertex $x$ of $C_3$.
We have established that $x$ is not adjacent to $v_4$ and has no neighbor in $B_4$; similarly, $x$ has no neighbor in $B_4'$, otherwise we would have a $(u_4, v_4)$-path on at most 5 vertices.
Moreover, $x$ has no neighbor in $C_3$ outside of the two vertices that are consecutive to it, otherwise we could replace $C_3$ with a shorter path.
Thus \[ |N(x) - (C \cup B_4 \cup B_4')| \geq \delta(H) - |N(x) \cap C| \geq \delta(H) - (3 + 5 + 2 + 1) = \delta(H) - 11 \geq 4. \]
But then \[ |H| \geq |C| + |B_4 \cup B_4'| + |N(x) - (C \cup B_4 \cup B_4')| \geq 17 + 10 + 4 = 31, \] contrary to the fact that $|H| \leq 30$.
\end{proof}

Let $C' = C_1 \cup C_2 \cup C_3$.
For any $a \in A_4$, we have \[ |A| \geq d(a) + 1 - |N(a) \cap C'| \geq \delta(H) + 1 - (3 + 3 + 3) = \delta(H) - 8 \geq 7. \]
Since there is no edge between $A$ and $B$, we have $v \notin A \cup B$, so $v$ is complete to $A \cup B$; more specifically, since every $(A, B)$-path must pass through $C'$, we must have $v \in C'$.
Even more specifically, since no two non-consecutive vertices on any $C_i$ can be adjacent by the minimality of $|C|$, if $v \in C_i$, then either $|C_i| = 2$ and $v$ is one of its endpoints or $|C_i| = 3$ and $v$ is its middle vertex, otherwise $C_i$ has a vertex that is not consecutive to and thus not adjacent to $v$, contrary to the definition of $v$.
Since $v$ is complete to $A$, $A$ can have no 7-clique.
So, if we let $A_0 = \{ a \in A : A - N[a] \neq \varnothing \}$, we must have $|A_0| \geq 2$, so that, for any $a \in A_0 - v_4$, we have \[ |A| \geq d(a) + 1 - |N(a) \cap C'| + 1 \geq 8. \]
If $|A| = 8$, then every vertex in $A_0$ has at most $|A| - 2 = 6$ neighbors in $A$, hence it has at least $\delta(H) - 6 \geq 9$ neighbors in $C'$.
This implies that each of $C_1, C_2, C_3$ has at least 3 vertices.
By minimality of $C$, any two nonconsecutive vertices in any of $C_1, C_2, C_3$ must be non-adjacent by minimality of $C$, so the vertex $v$ must be the central vertex of some $C_i$ such that $|C_i| = 3$.
Let $a, a' \in A_0 - u_4$.
Then $\{ a, a' \}$ is complete to $C_i$.
For any $b \in B$, since $v$ is adjacent to $b$, we can perform a $(v, a)$-reroute of $C_i$ and $C_4$, 
contrary to the minimality of $C$.
Thus $|A| \geq 9$, and, by symmetry, $|B| \geq 9$.

\begin{claim}
\label{claim:8neighbors}
Every vertex in $A - u_4$ and in $B - v_4$ has at most 7 neighbors in $C'$.
\end{claim}
\begin{proof}
We will prove that every vertex in $A - u_4$ has at most 7 neighbors in $C'$; the result for $B - v_4$ will follow by symmetry.

If there is some $a^* \in A - u_4$ with 9 neighbors in $C'$, then we must have $|C_i| \geq 3$ for every $i \in [3]$.
In that case, the vertex $v$ must be the central vertex of some $C_i$ such that $|C_i| = 3$.
But then we can perform a $(v, a^*)$-reroute of $C_i$ and $C_4$, contrary to the choice of $C$.
Thus every vertex in $A - u_4$ has at most 8 neighbors in $C'$; suppose some vertex in $A - u_4$ has exactly 8 neighbors in $C'$.
Note that taking every vertex in $A - A_0$, together with 1 vertex from $A_0$ and the vertex $v$, gives us a clique of order $|A - A_0| + 2$.
Since $H$ has no 8-clique, we have $|A - A_0| \leq 5$ and so $|A_0| \geq 4$.
For any $a, a' \in A_0 - u_4$ and any $b, b' \in B_0 - v_4$ (where $B_0 = \{ b \in B : B - N[b] \neq \varnothing \}$), by Lemma~\ref{l:si}(d), we have \[ \sum_{i=1}^4 (s_i + s_i') \geq 2(3 + 1 + 1) \geq 10, \] and by Lemma~\ref{l:si}(a), we have $s_4 + s_4' \leq 0$ and so \[ \sum_{i=1}^3 (s_i + s_i') \geq 10. \]
Label the indices in $[3]$ as $i, j, k$ so that $s_i + s_i' \geq s_j + s_j' \geq s_k + s_k'$.
Then $s_i + s_i' \geq \left\lceil \frac{10}{3} \right\rceil = 4$, so, by Lemma~\ref{l:si2}(a), $s_i + s_i' = 4$.
We then have $s_j + s_j' + s_k + s_k' \geq 6$, so $s_j + s_j' \geq 3$, implying that $s_k + s_k' \in \{ 2, 3, 4 \}$.
This implies that $|C_i|, |C_j| \in \{ 2, 3 \}$ and $|C_k| \in \{ 2, 3, 4, 5 \}$.

Note that, by Lemma~\ref{l:si2}(a), we must have \[ \sum_{i=1}^3 (s_i + s_i') \leq 12, \] which, by Lemma~\ref{l:si}(d), implies that $|A - N[a]| \leq 3$ for every $a \in A - u_4$.
We have $|N[u_4] \cap A| \geq d_H (u_4) + 1 - |C'| \geq 16 - 11 = 5$.
That is, $|N[u_4] \cap A| - |A - N[a]| \geq 2$, so either $a$ is adjacent to $u_4$ or $a$ and $u_4$ have at least 2 common neighbors.
This implies that $H[A]$ is 2-connected, and, by symmetry, $H[B]$ is 2-connected.

Suppose $|C_k| = 5$.
Then we must have $s_k = 1$ for every choice of $a \in A_0 - u_4$ and $b \in B_0 - u_4$.
This happens if and only if every vertex in $A_0 - u_4$ and in $B_0 - v_4$ has exactly 3 neighbors in $C_k$.
If we label the vertices of $C_k$ as $u_k x y z v_k$, then this implies that $y$ is complete to $(A_0 \cup B_0) - \{ u_4, v_4 \}$.
This, in turn, implies that no vertex in $(A_0 \cup B_0) - \{ u_4, v_4 \}$ is complete to $\{ x, y, z \}$; otherwise, if (for example) $a \in A_0 - u_4$ is complete to $\{ x, y, z \}$, then we can perform a $(y, a)$-reroute of $C_k$ and $C_4$. 
Then every vertex in $(A_0 \cup B_0) - \{ u_4, v_4 \}$ is complete to either $\{ u_k, x, y \}$ or $\{ y, z, v_k \}$.
By the pigeonhole principle, we may assume at least 2 vertices in $A_0 - u_4$ are complete to $\{ u_k, x, y \}$.
Then no vertex from $B$ can be adjacent to $x$, otherwise, if $a, a' \in A_0 - u_4$ are complete to $\{ u_k, x, y \}$, we can perform an $(x, a)$-reroute of $C_k$ and $C_4$.
Thus every vertex in $B_0 - v_4$ is complete to $\{ y, z, v_k \}$.
Note that, since $s_k = 1$ for all $a \in A_0 - u_4$ and $b \in B_0 - v_4$, we have \[ 10 \leq \sum_{i=1}^3 (s_i + s_i') \leq 10 \] for every choice of $a, a', b, b'$, so we have equality, which implies that $|B - N[b]| = 1$ for every $b \in B_0 - v_4$.
We claim that the graph $H[B \cup \{ v, y, z, v_k \} - v_4]$ is 4-linked (note that none of $y, z, v_k$ is adjacent to $u_k$, so none of these vertices is $v$).
The pairs of non-adjacent vertices in this graph are the pairs of non-adjacent vertices in $B_0 - v_4$, the pair $\{ y, v_k \}$, and possibly some pairs with one end in $\{ y, z, v_k \}$ and the other end in $B - B_0$.
The common neighbors of any pair of vertices in $B_0 - v_4$ include every other vertex in $B - v_4$ as well as $\{ v, y, z, v_k \}$, for a total of $4 + |B| - 3 = |B| + 1 \geq 10$ common neighbors.
The common neighbors of $y$ and $v_k$ include $v$ and $z$ as well as every vertex in $B_0 - v_4$, for a total of $|B_0 - v_4| + 2$ common neighbors.
The common neighbors of a vertex in $\{ y, z, v_k \}$ and a vertex in $B - B_0$ include every vertex in $B_0 - v_4$ as well as $v$, for a total of $|B_0 - v_4| + 1$ common neighbors.
By Corollary~\ref{c:uncommon}, if $|B_0 - v_4| + 1 \geq 7$, then this graph is 4-linked, so we must have $|B_0 - v_4| \leq 5$.
But then, since every vertex in $B_0$ is adjacent to all but exactly 1 vertex in $B$, we see that $B - v_4$ contains a clique of order \[ |B - v_4| - \left\lfloor \frac{|B_0 - v_4|}{2} \right\rfloor \geq |B| - 1 - 2 = |B| - 3 \geq 6. \]
Since $B - v_4$ has no 7-clique, we have equality here, so that $|B| = 9$ and $\left\lfloor \frac{|B_0 - v_4|}{2} \right\rfloor = 2$ exactly, and the largest clique in $|B_0 - v_4|$ is of order $|B_0 - v_4| - 2$: that is, $|B_0 - v_4| \in \{ 4, 5 \}$ and $\overline{H}[B_0 - v_4]$ has exactly 2 edges.
Note that $v_4$ is adjacent (in $H$) to both ends of these 2 edges; otherwise, any of these vertices that was not adjacent to $v_4$ would have 2 non-neighbors in $B$, a contradiction.
We can then form a clique out of every vertex in $B - (B_0 - v_4)$, exactly 2 vertices from $B_0 - v_4$, and the vertex $v$.
If $|B_0 - v_4| = 4$, then we have $|B - (B_0 - v_4)| = 5$, so this gives us an 8-clique.
Thus $|B_0 - v_4| = 5$.
In that case, we must have $|B_0| = 6$, with $v_4$ having exactly 1 non-neighbor in $B$.
Since $|B| = 9$, it follows that every vertex in $B - B_0$ has exactly 8 neighbors in $B$ and thus at least 7 neighbors in $C'$.
Note that, since we have $s_i + s_i' = s_j + s_j' = 4$ for all $a, a' \in A_0 - u_4$ and $b, b' \in B_0 - v_4$, and since $a^*$ has 8 neighbors in $C'$, at most 3 of which are in $C_k$, $a^*$ either has 3 neighbors in $C_i$ or 3 neighbors in $C_j$ (without loss of generality, the former).
Then, by Lemma~\ref{l:si2}(b), the middle vertex of $C_i$ has no neighbor in $B$, so every vertex in $B - B_0$ has at most 2 neighbors in $C_i$ and at most 3 neighbors in $C_j$, hence at least 2 neighbors in $C_k$.
We claim that, for every $b \in B - B_0$, $b$ has at least 2 neighbors in $\{ y, z, v_k \}$.
If not, since we have already observed that $x$ has no neighbor in $B$, the only way $b$ could have 2 neighbors in $C_k$ is if $N(b) \cap C_k = \{ u_k, y \}$.
Let $b' \in B_0$; then $bb' \in E(H)$ and $b'$ is complete to $\{ y, z, v_k \}$, so we can replace $C_k$ with $u_k b b' v_k$ and replace $C_4$ with a path through $a^*$ and $y$ and any vertex in $B_0 - b'$.
Also note that every vertex in $B$, including $v_4$, is anticomplete to $x$ and to the middle vertex of $C_i$, so, since $v_4$ has at least 8 neighbors in $C'$, it must be complete to $\{ y, z, v_k \}$.
When we turn our attention back to $H[B \cup \{ v, y, z, v_k \} - v_4]$, we now see that, given any nonadjacent pair consisting of a vertex in $\{ y, z, v_k \}$ and a vertex $b \in B_0$, it must be the case that $b$ is adjacent to the other two vertices of $\{ y, z, v_k \}$, so $b$ and its non-neighbor on this path have at least 1 common neighbor on this path.
They also have every vertex in $B_0$, including $v_4$, as a common neighbor, so this pair has at least $|B_0| + 2 = 8$ common neighbors.
As before, any non-adjacent pair of vertices in $B_0 - v_4$ has $|B| + 1 \geq 10$ common neighbors, and, since $u_4$ is now known to be a common neighbor of $y$ and $v_k$, those two vertices have $|B_0| + 2 \geq 8$ common neighbors.
The only time we can't apply Proposition~\ref{p:prop1} to show that the graph is 4-linked is in the case where our pairs are of the form $(y, b_1), (z, b_2), (v_k, b_3), (b, b')$, where $B - B_0 = \{ b_1, b_2, b_3 \}$ and $b, b' \in B_0$; in this one specific case, though, we have four vertices in $B_0$ that belong to none of the four pairs, so we can use those as the internal vertices to link all four pairs with paths of length 3.
In all other cases, we will have at most two pairs with an end in $\{ y, z, v_k \}$ and so we can find one pair each with 7, 8, 9, and 10 common neighbors.
Thus the graph is indeed 4-linked, a contradiction.

We claim that every vertex in $B - v_4$ has at most 6 neighbors in $C'$.
Recall that there is $a^* \in A - u_4$ that has 8 neighbors in $C'$.
Since $|N(a^*) \cap C_k| \leq 3$, we have $|N(a^*) \cap (C_i \cup C_j)| \geq 5$; we may assume $|N(a^*) \cap C_j| \geq 3$.
Then, since we have $s_j + s_j' \geq 3$, it follows from Lemma~\ref{l:si2} that $|C_j| = 3$ and the middle vertex of $C_j$ has no neighbor in $B$, so every vertex in $B$ has at most 2 neighbors in $C_j$.

Suppose $|C_k| = 4$.
Then $s_k + s_k' = 2$ for all choices of $a, a', b, b'$, implying that $s_i + s_i' = s_j + s_j' = 4$ for all choices of $a, a', b, b'$. 
Moreover, $v \notin C_k$ and $v \notin C_j$, so $v \in C_i$.
Since $s_i + s_i' = 4$ for all $a, a', b, b'$, every vertex in $(A_0 \cup B_0) - \{ u_4, v_4 \}$ is complete to $\{ u_j, v_j \}$ by Lemma~\ref{l:si2}(c), so it must be the case that $|C_i| = 2$; we may assume without loss of generality that $v_i = v$.
Then the vertex $a^*$ has at most 5 neighbors in $C_i \cup C_j$, so it has 3 neighbors in $C_k$; labeling the vertices of $C_k$ as $u_k x y v_k$, we may assume $a^*$ is complete to $\{ u_k, x, y \}$.
If any $b \in B - v_4$ has 3 neighbors in $C_k$, then $b$ must be adjacent to $x$.
In that case, $x$ has no other neighbors in $A$, otherwise we could perform an $(x, a^*)$-reroute of $C_k$ and $C_4$.
But then, for any $a \in A - a^*$, we have $N(a) \cap N(x) \subseteq C_i \cup C_j \cup \{ u_k, y, a^* \}$, so that \[ |N(a) \cup N(x)| = |N(a)| + |N(x)| - |N(a) \cap N(x)| \geq 2\delta(H) - 8, \] which implies \[ |B| \leq |H| - |N(a) \cup N(x)| \leq 2\delta(H) - 1 - \big( 2\delta(H) - 8 \big) = 7, \] a contradiction.
Thus no vertex in $B - v_4$ can have 3 neighbors in $C_k$.
Since we have shown that every vertex in $B$ has at most 2 neighbors in $C_j$ and $|C_i| \leq 2$, it follows that, in the case where $|C_k| = 4$, every vertex in $B$ has at most 6 neighbors in $C'$, as desired.

Now suppose $|C_k| \leq 3$, so that $\max \{ |C_i|, |C_j|, |C_k| \} = 3$.
If $\min \{ |C_i|, |C_j|, |C_k| \} = 2$ (say $|C_i| = 2$), then $a^*$ is complete to $C'$, having 3 neighbors in $C_j$ and 3 neighbors in $C_k$.
By Lemma~\ref{l:si2}(b), the middle vertex of $C_j$ is then anticomplete to $B$, so every vertex in $B$ has at most 2 neighbors in $C_j$.
If $s_k + s_k' \geq 3$ for some choice of $a, a', b, b'$, then the middle vertex of $C_k$ is also anticomplete to $B$, so every vertex in $B$ has at most 6 neighbors in $C'$, as desired.
If $s_k + s_k' = 2$ for every choice of $a, a', b, b'$, then necessarily $s_k = 1$ for every choice of $a, b$.
This implies that no vertex of $B - v_4$ has 3 neighbors in $C_k$, otherwise, for that choice of $b$ together with $a^*$, we would have $s_k = 3$.
Thus, we may assume $|C_i| = |C_j| = |C_k| = 3$.
As before, $a^*$ is complete to two of these paths, so the middle vertex of each of those two paths has no neighbor in $B$.
The third path contains exactly two neighbors of $a^*$, and its middle vertex must also be the vertex $v$.
Then no vertex of $B - v_4$ is complete to this third path, otherwise (if, say, $b \in B - v_4$ is the vertex in question) we would be able to perform a $(v, b)$-reroute of this path and $C_4$.

Now every vertex in $B - v_4$ has at most 6 neighbors in $C'$, hence at least $\delta(H) - 6 \geq 9$ neighbors in $B$, and $v_4$ has at most 7 neighbors in $C'$, hence at least 8 neighbors in $B$.
For any two non-adjacent vertices $b, b' \in B_0 - v_4$, the number of common neighbors of $b$ and $b'$ in $B$ is \[ |N(b) \cap N(b') \cap B| = |N(b) \cap B| + |N(b') \cap B| - |[N(b) \cup N(b')] \cap B| \geq 9 + 9 - (|B| - 2) = 20 - |B|. \]
These vertices then have $21 - |B|$ common neighbors in $B \cup \{ v \}$.
For any vertex $b \in B_0$ that is not adjacent to $v_4$, the number of common neighbors of $b$ and $v_4$ in $B$ is \[ |N(v_4) \cap B| + |N(b) \cap B| - |[N(v_4) \cup N(b)] \cap B| \geq 8 + 9 - (|B| - 2) = 19 - |B|, \] so these two vertices have at least $20 - |B|$ common neighbors in $B \cup \{ v \}$.
If $21 - |B| \geq 10$, then $H[B \cup \{ v \}]$ is 4-linked by Corollary~\ref{c:uncommon}, so we must have $21 - |B| \leq 9$ and thus $|B| \geq 12$.

Since $|C'| \geq 8$, $|B| \geq 12$, and $|H| \leq 30$, we must have $|A| \leq 10$.
If $|H| = 30$, then we have $2\delta(H) \geq |H| + 2$, and we either have $|A| = 10$ and $|C'| = 8$ or $|A| = 9$ and $|C'| = 9$.
If $|A| = 9$ and $|C'| = 9$, then every vertex in $A$ has at most 9 neighbors in $C'$, hence at least 6 neighbors in $A$; if $|A| = 10$ and $|C'| = 8$, then every vertex in $A$ has at most 8 neighbors in $C'$, hence at least 7 neighbors in $A$.
Either way, $\Delta(\overline{H}[A]) \leq 2$.
By Lemma~\ref{l:si}(d), if we take a vertex $a''$ with degree 2 in $\overline{H}[A]$ and a vertex $b''$ with degree 1 in $\overline{H}[B]$, we will have $\sum_{i=1}^3 s_i'' \geq (|H| + 2) - (|H| - 2) + 2 + 1 = 7$, so that $s_i'' = 3$ for some $i \in [3]$.
Then, for this choice of $i$, $|C_i| = 3$ and $\{ a'', b'' \}$ is complete to $C_i$.
But then the middle vertex $x$ of $C_i$ can have no neighbor in $(A \cup B) \setminus \{ a'', b'' \}$, otherwise we can perform an $(x, a'')$- or $(x, b'')$-reroute of $C_i$ and $C_4$.
Then $x$ has at most 7 neighbors in $C_i$ and 2 neighbors in $A \cup B$, so that $d(x) \leq 9 < \delta(H)$, a contradiction.
Thus we must have $\Delta(\overline{H}[A]) = 1$; similarly, $\Delta(\overline{H}[B]) = 1$.
Then, for every $a'' \in A_0$ and every $b'' \in B_0$, we have $\sum_{i=1}^3 s_i'' \geq 6$; we have seen that we get a contradiction if $s_i'' = 3$ for any $i \in [3]$, so we must have $s_1'' = s_2'' = s_3'' = 2$.
This implies that $A_0$ and $B_0$ are complete to $\{ u_1, v_1, u_2, v_2, u_3, v_3 \}$.
But $|B_0| = 12$, so $B_0$ contains a 6-clique.
By taking any vertex from $\{ u_1, v_1, u_2, v_2, u_3, v_3 \} \setminus v$, together with $v$ and with a 6-clique in $B_0$, we get an 8-clique, a contradiction.
Thus $|H| \neq 30$.
We must then have $|H| = 29, |B| = 12, |C'| = 8$, and $|A| = 9$.
Then every vertex in $A$ has at most $8$ neighbors in $C'$, hence at least 7 neighbors in $A$, so every vertex in $\overline{H}[A]$ has degree 1.
Moreover, every vertex in $A_0$ has exactly 7 neighbors in $A$ and is thus complete to $C'$, while every vertex in $A - A_0$ has 8 neighbors in $A$ and is thus complete to every vertex in $C'$ but one.
Choose $i \in [3]$ such that $|C_i| = 3$ and $v \notin C_i$; write $C_i = u_i x v_i$.
We claim that $H[A \cup C_i \cup \{ v \}]$ is a 4-linked graph.
The pairs of non-adjacent vertices in this graph are the pair $\{ u_i, v_i \}$, some number of pairs of vertices in $A_0$, and some number of pairs of vertices with one end in $C_i$ and the other end in $A - A_0$.
The vertices $u_i$ and $v_i$ have every vertex in $A_0$, as well as $x$ and $v$, as common neighbors, for a total of $|A_0| + 2$ common neighbors.
Any pair of non-adjacent vertices in $A_0$ is complete to every other vertex in $A \cup C_i \cup \{ v \}$, so they have $|A| + 2 \geq 11$ common neighbors.
Given a vertex $a \in A - A_0$ that is not complete to $C_i$, we know that $a$ is adjacent to 2 of the 3 vertices of $C_i$.
If the vertex that $a$ is not adjacent to is $x$, then the common neighbors of $a$ and $x$ include every vertex of $A_0$ as well as $u_i, v_i$, and $v$, for a total of $|A_0| + 3$ common neighbors.
If the vertex that $a$ is not adjacent to is an endpoint of $C_i$, then the common neighbors of $a$ and that endpoint include every vertex of $A_0$ as well as $x$ and $v$, for a total of $|A_0| + 2$ common neighbors.
So, given any 4 pairs of non-adjacent vertices in this graph, we have at most 3 pairs with at least $|A_0| + 2$ common neighbors, with the fourth pair necessarily having 11 common neighbors.
By Proposition~\ref{p:prop1}, if $|A_0| + 2 \geq 9$, then this graph is 4-linked, so we may assume $|A_0| \leq 6$.
On the other hand, since $A \cup \{ v \}$ has a clique of order $|A - A_0| + \left\lceil \frac{|A_0|}{2} \right\rceil + |\{ v \}| = 10 - \left\lfloor \frac{|A_0|}{2} \right\rfloor$ and has no 8-clique, we must have $\left\lfloor \frac{|A_0|}{2} \right\rfloor \geq 3$, so that $|A_0| = 6$.
Each of the 3 vertices in $A - A_0$ is adjacent to all but 1 of the 7 vertices in $C' - v$, so there is a vertex in $y \in C' - v$ that is adjacent to all 3 of them.
But then taking a 3-clique in $A_0$, together with all 3 vertices of $A - A_0$, $v$, and $y$, gives us an 8-clique, a contradiction.
\end{proof}

Now every vertex in $A - u_4$ has at most 7 neighbors in $C'$, hence at least 8 neighbors in $A$, so that $|A| \geq 9$.
Since $A - u_4$ is not an 8-clique, $A_0 - u_4$ is non-empty, so that every vertex in $A - u_4$ has at least 8 neighbors and at least 1 non-neighbor in $A$, implying that $|A| \geq 10$.
If there are $a \in A - u_4$ and $b \in B - u_4$ such that $|A - N[a]| + |B - N[b]| \geq 4$, then, by Lemma~\ref{l:si}(d), we would have $\sum_{i=1}^4 s_i \geq 7$, implying that there is $i \in [3]$ such that $s_i \geq 3$.
In that case, though, the middle vertex $x$ of $C_i$ could have no neighbor in $(A \cup B) - \{ a, b \}$; otherwise, if it had a neighbor $a' \in A - a$, we could perform an $(x, a$)-reroute of $C_i$ and $C_4$.
But then we would have $d(x) \leq |H| - |A \cup B| - |\{ x \}| \leq 30 - 20 - 1 = 9 < \delta(H)$, a contradiction.
Thus $|A - N[a]| + |B - N[b]| \leq 3$ for all $a \in A - u_4$ and $b \in B - u_4$, so we may assume without loss of generality that every vertex in $\overline{H}[A]$, except possibly for $u_4$, has degree at most 1.
Let $d$ be the degree of $u_4$ in $\overline{H}[A]$.
Then, in the graph $H[A \cup \{ v \}]$, every pair of non-adjacent vertices that includes $u_4$ has $(|A| + 1 - 2) - (d - 1) = |A| - d$ common neighbors, and every pair of non-adjacent vertices that does not include $u_4$ has $|A| - 1$ common neighbors.
Since $H[A \cup \{ v \}$ is not 4-linked, by Corollary~\ref{c:uncommon}, we either have $|A| - 1 \leq 9$ or $|A| - d \leq 6$.
But we can also note that any vertex in $A - N[u_4]$ is complete to $A - u_4$ since it only has degree 1 in $\overline{H}[A]$, so taking every non-neighbor of $u_4$ (all $d$ of them), together with the largest possible clique in $N(u_4) \cap A$ (which has order at least $\left\lceil \frac{|A| - d - 1}{2} \right\rceil$) and $v$, gives us a clique of order $\left\lceil \frac{|A| + d + 1}{2} \right\rceil$, so we must have $|A| + d + 1 \leq 14$.
This implies $-d \geq |A| + 1 - 14$, so that $|A| - d \geq 2|A| + 1 - 14 \geq 7$.
Therefore, we cannot have $|A| - d \leq 6$, so we must have $|A| - 1 \leq 9$, implying that $|A| = 10$ exactly.
Then every vertex in $A_0 - u_4$ has exactly 8 neighbors in $A$ and thus exactly 7 neighbors in $C'$.
If $|A_0 - u_4| \leq 3$, then taking 2 adjacent vertices in $A_0 - u_4$, together with every vertex of $A - (A_0 \cup \{ u_4 \})$, of which there are at least 6, gives us an 8-clique, so we must have $|A_0 - u_4| \geq 4$.

\begin{claim}
There is a subgraph of $H[C']$ that is isomorphic to $P_3$ and complete to $A_0 - u_4$.
\end{claim}
\begin{proof}
Suppose that $|C_k| = 5$ for some $k \in [3]$.
Arguing as in Claim~\ref{claim:8neighbors}, 
we must have $\sum_{i=1}^3 s_i \geq 10$, and so we must have $s_k = 1$ for every choice of $a$ and $b$, which implies that every vertex of $(A_0 \cup B_0) - \{ u_4, v_4 \}$ has 3 neighbors in $C_k$.
If we write $C_k = u_k x y z v_k$, then $y$ is complete to $(A_0 \cup B_0) - \{ u_4, v_4 \}$, so no vertex of $(A_0 \cup B_0) - \{ u_4, v_4 \}$ is complete to $\{ x, y, z \}$, otherwise, calling that vertex $w$, we can perform a $(y, w)$-reroute of $C_k$ and $C_4$.
Thus every vertex in $(A_0 \cup B_0) - \{ u_4, v_4 \}$ is complete to either $\{ u_k, x, y \}$ or $\{ y, z, v_k \}$.
By the pigeonhole principle, we may assume without loss of generality that there are $a, a' \in A_0 - u_4$ are both complete to $\{ u_k, x, y \}$.
Then no vertex in $B - v_4$ is complete to $\{ u_k, x, y \}$, otherwise we can perform an $(x, a)$-reroute of $C_k$ and $C_4$.
Thus every vertex in $B_0 - v_4$ is complete to $\{ y, z, v_k \}$, and a symmetrical argument shows that no vertex of $A_0 - u_4$ is complete to $\{ y, z, v_k \}$, so that $H[\{ u_k, x, y \}]$ is a $P_3$ that is complete to $A_0 - u_4$, as desired.

We may assume $\max \{ |C_1|, |C_2|, |C_3| \} \leq 4$.
By the pigeonhole principle, there exist $i \in [3]$ and $a, a' \in A_0 - u_4$ such that $a$ and $a'$ each have 3 neighbors in $C_i$.
Let $x$ be the middle vertex of the three neighbors of $a$ on $C_i$.
Then $a'$ must be adjacent to $x$ as well.
This implies that $x$ has no neighbor in $B$, otherwise we can perform an $(x, a)$-reroute of $C_i$ and $C_4$.
Then no vertex in $B$ has 3 consecutive neighbors on $C_i$, so that every vertex in $B$ has at most 2 total neighbors on $C_i$.
If $B$ has at least 3 vertices that each have 7 neighbors in $C'$, then each of these 3 vertices must have at least 5 neighbors in $C' - C_i$, so that there exist $j \in [3] - i$ and $b, b' \in B$ such that $b$ and $b'$ each have 3 neighbors in $C_j$; we may assume $b \neq v_4$.
A symmetrical argument shows that some vertex in $C_j$ has no neighbor in $A$.
Then every vertex in $A_0 - u_4$ has 7 neighbors in $C'$ and at most 2 neighbors in $C_j$, so it has at least 5 neighbors in $C - C_j$.
The vertex $v$ then cannot belong to $C_i$ or $C_j$, as each of these paths has an internal vertex that is not complete to $A \cup B$; let $C_k$ be the path of $C'$ that contains $v$.
Then no vertex of $(A \cup B) - \{ u_4, v_4 \}$ can have 3 neighbors in $C_k$, otherwise, calling that vertex $w$, we could perform a $(v, w)$-reroute of $C_k$ and $C_4$.
Therefore, every vertex of $A - u_4$ has at most 2 neighbors in $C_j$ and at most 2 neighbors in $C_k$, hence every vertex of $A - u_4$ has 3 neighbors in $C_i$.
If $|C_i| = 3$ or if every vertex of $A - u_4$ has the same 3 neighbors on $C_i$, then we are done.
If not, then $|C_i| = 4$ and we can write $C_i = u_i x y v_i$, where some $a \in A_0 - u_4$ is complete to $\{ u_i, x, y \}$ and some $a' \in A_0 - u_4$ is complete to $\{ x, y, v_i \}$.
We previously observed that $x$, the middle of the three neighbors of $a$, has no neighbor in $B$, and the same argument shows that $y$, the middle of the three neighbors of $a'$, has no neighbor in $B$ either.
By Lemma~\ref{l:si}(b), no vertex of $B - v_4$ is adjacent to both $u_i$ and $v_i$, so every vertex in $B - v_4$ has at most 1 neighbor in $C_i$, at most 3 neighbors in $C_j$, and at most 2 neighbors in $C_k$, contrary to our assumption that 3 vertices in $B$ each have 7 neighbors in $C'$.

Now suppose at most 2 vertices in $B$ have 7 neighbors in $C'$.
If we call these vertices $b$ and $b'$, then we have \[ |N(b) \cap N(b') \cap B| = |N(b) \cap B| + |N(b') \cap B| - |[N(b) \cup N(b')] \cap B| \geq 8 + 8 - (|B| - 2) = 18 - |B|. \]
Every other vertex in $B$ has at most 6 neighbors in $C'$ and thus at least 9 neighbors in $B$.
If $b''$ is one of these vertices, then the number of common neighbors of $b$ and $b''$ in $B$ is at least \[ 8 + 9 - (|B| - 2) = 19 - |B|, \] and the number of common neighbors of any two vertices in $B$ other than $b$ and $b'$ is at least \[ 9 + 9 - (|B| - 2) = 20 - |B|. \]
The number of common neighbors each of these pairs has in $B \cup v$ will be 1 more than this.
By Proposition~\ref{p:prop1}, if $20 - |B| + 1 \geq 10$, then $H[(B - v_4) \cup v]$ would be 4-linked, a contradiction, so we must have $21 - |B| \leq 9$ and thus $|B| \geq 12$.
Since $|B| \geq 12$, $|A| = 10$, and $|H| \leq 30$, we have $|C'| \leq 8$; since every vertex in $A_0 - u_4$ has exactly 7 neighbors in $C'$, we have $|C'| \in \{ 7, 8 \}$.
If $A_0 - u_4$ is complete to $C'$, then there is some $i \in [3]$ such that $|C_i| = 3$ and $A_0 - u_4$ is complete to $C_i$, as desired.
Thus we may assume that $A_0 - u_4$ is not complete to $C'$, which implies $|C'| = 8$ and so $|H| = 30$.
We then have $2\delta(H) \geq |H| + 2$, so, by Lemma~\ref{l:si}(d), for any $a \in A_0$ and $b \in B_0$, we have $\sum_{i=1}^3 s_i \geq 6$.
If $s_i = 3$ for any $i \in [3]$, then $|C_i| = 3$ and $\{ a, b \}$ is complete to $C_i$.
The middle vertex $x$ of $C_i$ then has no neighbors in $(A \cup B) \setminus \{ a, b \}$, otherwise we can perform an $(x, a)$- or $(x, b)$-reroute of $C_i$ and $C_4$.
But then $d(x) \leq |C'| - 1 + 2 = 9 < \delta(H)$, a contradiction.
Thus $s_1 = s_2 = s_3 = 2$; since we have equality, we must also have $\Delta(\overline{H}[A]) = \Delta(\overline{H}[B]) = 1$.
We have $2 \leq |C_1| \leq \left\lfloor \frac{|C'|}{3} \right\rfloor = 2$, so $C_1$ is a $K_2$ that is complete to $B_0$.
If $|B_0| = 12$, then $B_0$ contains a 6-clique and so $B_0 \cup C_1$ contains an 8-clique, a contradiction.
But if $|B_0| < 12$, then $B$ contains a 7-clique and so $B \cup \{ v \}$ contains an 8-clique, giving us another contradiction.
\end{proof}

Now we have a $P_3$ that is complete to $A_0$; label its vertices $x_1 y x_2$.
Note that $v \neq x_1$ and $v \neq x_2$, because $x_1 x_2 \notin E(H)$, and $v \neq y$, because our arguments above showed that, no matter what vertex $y$ is chosen to be, it can have no neighbor in $B - v_4$.
We can thus consider the graph $H[A \cup \{ x_1, y, x_2, v \}]$; we claim that this graph is 4-linked.

The pairs of non-adjacent vertices in this graph are the pair $\{ x_1, x_2 \}$, some number of pairs of vertices in $A_0$, and some number of pairs of vertices with one end in $\{ x_1, y, x_2 \}$ and the other end in $A - A_0$.
The vertices $x_1$ and $x_2$ have every vertex in $A_0$, as well as $y$ and $v$, as common neighbors, for a total of $|A_0| + 2$ common neighbors.
Any pair of non-adjacent vertices in $A_0$ is complete to every other vertex in $A \cup \{ x_1, y, x_2, v \}$, so they have $|A| + 2 = 12$ common neighbors.
Given a vertex $a \in A - A_0$ that is not complete to $C_i$ 
the common neighbors of $a$ and any non-adjacent vertex on the $P_3$ include every vertex of $A_0$ as well as $v$, for a total of $|A_0| + 1$ common neighbors.
So, given any 4 pairs of non-adjacent vertices in this graph, we have at most 3 pairs with at least $|A_0| + 1$ common neighbors, with the fourth pair necessarily having 12 common neighbors.
By Proposition~\ref{p:prop1}, if $|A_0| + 1 \geq 9$, then this graph is 4-linked, so we may assume $|A_0| \leq 7$.
At the same time, $A \cup \{ v \} - \{ u_4 \}$ has a clique of order $|A| - 1 - \left\lfloor \frac{|A_0|}{2} \right\rfloor + 1 = 10 - \left\lfloor \frac{|A_0|}{2} \right\rfloor$; since $H$ has no 8-clique, we must have $\left\lfloor \frac{|A_0|}{2} \right\rfloor \geq 3$, so that $|A_0| \geq 6$.
Since $A_0$ is complete to the clique $\{ x_1, y, v \}$, $A_0$ can have no 5-clique, so $\overline{H}[A_0]$ has at least 2 edges.
If $|A_0| = 6$, then $A_0 \cup u_4$ contains a 4-clique: if $\overline{H}[A_0]$ has 2 edges, then $A_0$ itself has a 4-clique, and if $\overline{H}[A_0]$ has 3 edges, then taking 1 vertex from each edge together with $u_4$ gives us the 4-clique.
This 4-clique, together with the 3 vertices of $A - (A_0 \cup u_4)$ as well as $v$, gives us an 8-clique, a contradiction, so we have $|A_0| = 7$.
If $\overline{H}[A_0]$ has 2 edges, then $A_0$ has a 5-clique, which, together with the 2 vertices of $A - (A_0 \cup u_4)$ as well as $v$, give us an 8-clique, so $\overline{H}[A_0]$ has 3 edges, meaning one vertex of $A_0$ has $u_4$ as its non-neighbor and each of the other six vertices of $A_0$ has its non-neighbor in $A_0$.
Now consider the graph $H[A \cup \{ v, x_1 \}]$.
The non-adjacent pairs in this graph consist of four pairs of vertices in $A_0$ (one of which includes $u_4$), up to 2 pairs of vertices composed of $x_1$ together with a vertex in $A - A_0$, and possibly the pair $\{ u_4, x_1 \}$.
Let $y_1, y_2, y_3, y_4, z_1, z_2, z_3, z_4$ be distinct vertices in this graph such that $y_1 z_1, y_2 z_2, y_3 z_3, y_4 z_4 \notin E(H)$.
If $x_1$ is one of these vertices, we may assume it equals $y_1$.
In that case, we can link $y_1$ to $z_1$ with $v$; the remaining pairs must all be from $A_0$, so they are all complete to the 3 vertices in $(A \cup \{ v, x_1 \}) - \{ y_1, y_2, y_3, y_4, z_1, z_2, z_3, z_4, v \}$, so we have ample freedom to link those three remaining pairs with paths of length 3.
If $x_1$ does not belong to the set of eight vertices, then each of the pairs comes from $A_0$: the pair that includes $u_4$ has 3 common neighbors outside of $A_0$ (the two vertices of $A - A_0$ as well as $v$) and every pair that does not include $u_4$ has 4 common neighbors outside of $A_0$ (the same 3 common neighbors as the pair including $u_4$, as well as $x_1$), so each pair can be linked with a path of length 3.
Thus $H$ has a 4-linked subgraph.
\end{proof}

\section{Appendix 2}

In this appendix, we prove Lemma~\ref{knitted1}(c) with two lemmas. Lemma~\ref{l:app2part1} covers cases (ii) and (iii) of Lemma~\ref{knitted1}, and Lemma~\ref{l:app2part2} covers case (i).

\begin{lemma}
\label{l:app2part1}
Let $H$ be a graph, $v \in V(H)$ such that $H = N[v]$.
Suppose $\delta(H) \geq 9$ and $|H| \leq 16$.
Suppose further that, if $15 \leq |H| \leq 16$, then $H$ has at most 2 vertices of degree 9, and if it has 2 such vertices, they are not adjacent.
Then $H$ has a $(2,2,2,1)$-knitted subgraph.
\end{lemma}
\begin{proof}
Suppose $H$ has no $(2,2,2,1)$-knitted subgraph.
This implies that $H$ has no 7-clique. 
Since $H$ itself is not $(2,2,2,1)$-knitted, we will define $u_0, u_1, v_1, u_2, v_2, u_3, v_3$ as in the proof of Lemma~\ref{l:23case}, and define $C = C_0 \cup C_1 \cup C_2 \cup C_3 \subseteq V(H)$ as follows:
\begin{enumerate}[(i)]
    \item $C_0 = \{ u_0 \}$
    \item If the graph $H \setminus C_0 \cup \{ u_1, v_1, u_2, v_2, u_3, v_3 \}$ has a $(u_1, v_1)$-path on at most 4 vertices, then $C_1$ is the vertex set of this path. Otherwise, $C_1 = \{ u_1, v_1 \}$.
    \item If the graph $H \setminus (C_0 \cup C_1 \cup \{ u_2, v_2, u_3, v_3 \})$ has a $(u_2, v_2)$-path on at most 4 vertices, then $C_2$ is the vertex set of this path. Otherwise, $C_2 = \{ u_2, v_2 \}$.
    \item $C_3 = \{ u_3, v_3 \}$.
	\item Subject to (i)-(iv), rearranging the pairs $(u_1, v_1), \dotsc, (u_3, v_3)$ if necessary, as many of the $C_i$ as possible induce connected subgraphs of $H$.
	\item Subject to (i)-(v), rearranging the pairs $(u_1, v_1), \dotsc, (u_3, v_3)$ if necessary, $C$ has as few vertices as possible.
\end{enumerate}
We may again assume, rearranging if necessary, that there is $t \in \{ 0, 1, 2, 3 \}$ such that $H[C_i]$ is connected for all $i \leq t$ and that $|C_i| \leq |C_j|$ whenever $i < j \leq t$.
Note that, for any $x, y \in V(H)$, we have $d(x) + d(y) \geq |H| + 4$ unless $x$ or $y$ has degree 9.
\begin{claim}
$t \geq 1$. 
\end{claim}
\begin{proof}
Let $A_1 = N(u_1) \setminus C$.
We have $|N(u_1) \cap C| \leq | \{ u_0, u_2, v_2, u_3, v_3 \} |$, so $|A_1| \geq d(u_1) - 5 \geq 4$.
Let $a \in A_1$; if $|H| \geq 15$, we may choose $a$ to be a vertex of degree at least 10.
Then \[ |N(a) \cap N(v_1)| = d(a) + d(v_1) - |N(a) \cup N(v_1)| \geq |H| + 3 - (|H| - 2) = 5. \]
Note that $a$ has at most 3 neighbors in $C$: $u_0$ and at least 1 vertex each in $C_2$ and $C_3$ (it cannot be adjacent to $v_1$ or both vertices in $C_2$ or $C_3$. otherwise we get a path of length 3 in $C$).
Then $v_1$ and $a$ have a common neighbor $x \in V(H) \setminus C$, so we get a path $u_1 a x v_1$.
\end{proof}
\begin{claim}
$t \geq 2$. 
\end{claim}
\begin{proof}
Let $A_2 = N(u_2) \setminus C$.
We have $|N(u_2) \cap C| \leq |C_0| + |C_1| + |C_3| \leq 1 + 4 + 2 = 7$, so $A_2 \geq 2$.
Let $a \in A_2$; if $|H| \geq 15$, we may choose $a$ to be a vertex of degree at least 10.
Then \[ |N(a) \cap N(v_2)| = d(a) + d(v_2) - |N(a) \cup N(v_2)| \geq |H| + 3 - (|H| - 2) = 5. \]
If $a$ and $v_2$ have a common neighbor in $H \setminus C$, then we can connect $C_2$ with a path of length 4, so we may assume every common neighbor of $a$ and $v_2$ is in $C$.
Since $a$ has at most 1 neighbor in $C_0$ and at most 1 neighbor in $C_3$ and since $v_2$ has no neighbors in $C_2$, $a$ and $v_2$ must have at least 3 common neighbors in $C_1$.
By Lemma~\ref{l:si}(b), these 3 neighbors must be consecutive; let $x$ be the vertex in the middle of these three.
Then, letting $A$ be the component of $H \setminus (C_0 \cup C_1 \cup C_3)$ that contains $u_2$, we see that $x$ has no neighbor in $A \setminus a$, otherwise we can perform an $(x, a)$-reroute of $C_1$ and $C_2$.
In particular, $u_2$ is not adjacent to $x$, so that $|N(u_2) \cap C| \leq 6$ and thus $A_3 \geq 3$, so we can choose a vertex $a' \in A_2 \setminus a$, again chosen so that $d(a') \geq 10$ if $|H| = 16$.
The same argument then shows that $v_2$ and $a'$ must have 3 common neighbors in $C_1$, but $|C_1| \leq 4$ and $C_1$ has an interior vertex that is not adjacent to $a'$, a contradiction.
\end{proof}
We may now assume $t=2$.
We define $A_3, B_3, A$, and $B$ as in Lemma~\ref{l:si}.
Letting $C' = C_0 \cup C_1 \cup C_2$, we proceed by cases according to the number of vertices in $C'$.
\begin{case}
$|C'| = 5$.
\end{case}
\begin{proof}
We have $|N(u_3) \cap C'| \leq |C'| = 5$ and so $A_3 \geq d(u_3) - 5 \geq 4$.
We thus have $|A| \geq 5$; by symmetry, $|B| \geq 5$.
Since $|H| \leq 16$, we have $|A \cup B| \leq |H| - |C'| \leq 11$, so we may assume $|A| = 5$ exactly.
Then every vertex in $A$ has at most 4 neighbors in $A$ and thus at least 5 neighbors in $C'$; that is, $A$ is complete to $C'$.
But then $A$ is a 5-clique, so taking $A$ together with any 2 consecutive vertices in $C'$ gives us a 7-clique, a contradiction.
\end{proof}
\begin{case}
$|C'| = 6$.
\end{case}
\begin{proof}
We have $|N(u_3) \cap C'| \leq |C'| = 6$ and so $A_3 \geq d(u_3) - 6 \geq 3$.
We thus have $|A| \geq 4$; by symmetry, $|B| \geq 4$.
If $|A| = 4$, then every vertex in $A$ has at most 3 neighbors in $A$, hence at least 6 neighbors in $C'$; that is, $A$ is complete to $C'$.
But then we can take $A$ together with $v$ and any two adjacent vertices in $C' \setminus v$ to get a 7-clique, a contradiction.
Thus $|A| \geq 5$ and $|B| \geq 5$; since $|H| \leq 16$, we must have $|A| = |B| = 5$, so that $|H| = 16$ and thus every vertex in $H$ except for at most two has degree at least 10.
We may then assume that $A$ does not have more vertices of degree 9 than $B$ does.
Then 4 of the 5 vertices of $A$ have at most 4 neighbors in $A$ and at least 6 neighbors in $C'$; that is, $A$ is a 5-clique, and every vertex of $A$ except for at most 1 is complete to $C'$.
We may again take a 4-clique in $A$, together with $v$ and 2 adjacent vertices in $C' \setminus v$, to get a 7-clique, a contradiction.
\end{proof}
\begin{case}
$|C'| = 7$.
\end{case}
\begin{proof}
We have $|N(u_3) \cap C'| \leq |C'| = 7$ and so $A_3 \geq d(u_3) - 7 \geq 2$.
We thus have $|A| \geq 3$; by symmetry, $|B| \geq 3$.
Let $a \in A \setminus u_3$ and $b \in B \setminus v_3$. 
If we can choose vertices $a$ and $b$ such that $d(a) + d(b) \geq |H| + 4$, then, by Lemma~\ref{l:si}(d), we have $\sum_{i=0}^3 s_i \geq 6$.
If we cannot, then either every vertex in $A \setminus u_3$ or every vertex in $B \setminus v_3$ (without loss of generality, the former) has degree 9.
Since the two vertices of degree 9 must be non-adjacent, we can choose $a \in A \setminus u_3$ that has a non-neighbor in $A$, so, by Lemma~\ref{l:si}(d), we still have $\sum_{i=0}^3 s_i \geq 6$.
Then either $s_1 = 3$ or $s_2 = 3$; without loss of generality, the former.
But then $|C_1| = 3$ and $a$ and $b$ are both complete to $C_1$.
This implies that the middle vertex $x$ of $C_1$ has no neighbor in $(A \cup B) \setminus \{ a, b \}$ (otherwise, we can perform an $(x, a)$- or $(x, b)$-reroute of $C_1$ and $C_3$).
But $x$ has at most $|C'| - 1 = 6$ neighbors in $C'$, so, if it has at most 2 neighbors in $A \cup B$, then $d(x) \leq 8 < \delta(H)$, a contradiction.
\end{proof}
\begin{case}
$|C'| = 8$.
\end{case}
\begin{proof}
We must have $|C_1| = 3$ and $|C_2| = 4$.
We have $|N(u_3) \cap C'| \leq |C'| = 8$ and so $A_3 \geq d(u_3) - 8 \geq 1$.
For any $a \in A_3$, $|N(a) \cap C'| \leq 1 + 3 + 3 = 7$, so we get $|A| \geq |N(a) \setminus C'| + 1 \geq 9 - 7 + 1 = 3$.
We thus have $|A| \geq 3$; by symmetry, $|B| \geq 3$.
If $|H| \leq 14$, we must then have $|H| = 14$ and $|A| = |B| = 3$.
Then every vertex in $A$ and in $B$ has at least 7 neighbors in $C'$; in particular, $A \setminus u_3$ and $B \setminus v_3$ are complete to $C_1$.
But then we can perform a reroute of $C_1$ and $C_3$, a contradiction.
Thus $|H| \geq 15$, which implies $H$ has at most 2 vertices of degree 9.
Since $|A| \geq 3$, $A \setminus u_3$ either has a vertex of degree 10 or consists of the two non-adjacent vertices of degree 9; since this vertex has at most 7 neighbors in $C'$, it either has 3 neighbors (if its degree is 10) or 2 neighbors and 1 non-neighbor (if its degree is 9) in $A$, so that $|A| \geq 4$.
By symmetry, $|B| \geq 4$, so, since $|H| \leq 16$, we must have $|A| = |B| = 4$.
Then each of $A \setminus u_3$ and $B \setminus v_3$ must have a vertex of degree 10, say $a$ and $b$, which have at most 3 neighbors in $A$ and $B$, respectively, and thus have 7 neighbors in $C'$; in particular, these vertices are both complete to $C_1$.
But then the middle vertex $x$ of $C_1$ has no neighbor in $(A \cup B) \setminus \{ a, b \}$, otherwise we can perform an $(x, a)$- or $(x, b)$-reroute of $C_1$ and $C_3$.
Of the remaining 4 vertices in $(A \cup B) \setminus \{ u_3, v_3, a, b \}$, at least 2 have degree at least 10; we may assume some $a' \in A \setminus \{ u_3, a \}$ has degree 10.
Then $a'$ has at most 3 neighbors in $A$ and thus at least 7 neighbors in $C'$, so it is complete to $C_1$, a contradiction.
\end{proof}
\begin{case}
$|C'| = 9$.
\end{case}
\begin{proof}
We must have $|C_1| = 4$ and $|C_2| = 4$.
It is now possible that $A_3$ or $B_3$ is empty.
Suppose $A_3$ is empty; then $u_3$ has degree 9 and is complete to $C'$.
It follows that $v_3$ is anticomplete to the interior vertices of $C_1$ and $C_2$, otherwise we can connect $C_3$ with a path of length 3 to get a choice of $C$ with fewer vertices.
Then $v_3$ has at most 5 neighbors in $C'$, so that $|B_3| \geq 4$; thus $|H| \geq |\{ u_3 \}| + |B| + |C'| \geq 15$, so that $H$ has at most 2 vertices of degree 9, including $u_3$.
Note that $|B| \leq |H| - |C' \cup \{ u_3 \}| \leq 6$.
Then every vertex in $B$ that has degree 10 has at most 5 neighbors in $B$ and thus at least 5 neighbors in $C'$.
We observe that no $b \in B \setminus v_3$ can have 3 neighbors on $C_1$ or on $C_2$; if, say, $b$ has 3 neighbors on $C_1$ and $x$ is the vertex in the middle of those 3 neighbors, then $x$ has no neighbor in $B \setminus b$, otherwise we can perform an $(x, b)$-reroute of $C_1$ and $C_3$.
Thus every vertex in $B \setminus v_3$ must have at most 1 neighbor in $C_0$, at most 2 neighbors in $C_1$, and at most 2 neighbors in $C_2$.
Each of these vertices then has at least 5 neighbors in $B$.
Moreover, since $v_3$ also has at most 5 neighbors in $C'$ (the sole vertex of $C_0$ and the endpoints of $C_1$ and $C_2$), it also has 5 neighbors in $B$; since $B \setminus v_3$ has at most 1 vertex of degree 9, $B$ must be a 6-clique.
But then $B \cup \{ v \}$ is a 7-clique, a contradiction.

Thus we may assume $A_3 \neq \varnothing$ and $B_3 \neq \varnothing$.
Every vertex in $A_3$ has at most 7 neighbors in $C'$ and thus at least 2 neighbors in $A$, so that $|A| \geq 3$ and, by symmetry, $|B| \geq 3$.
Then $|A \cup B \cup C'| \geq 15$, so $H$ has at most 2 vertices of degree 9.
Since $|A| \geq 3$, $A \setminus u_3$ either has a vertex of degree 10 or consists of the two non-adjacent vertices of degree 9; since this vertex has at most 7 neighbors in $C'$, it either has 3 neighbors (if its degree is 10) or 2 neighbors and 1 non-neighbor (if its degree is 9) in $A$, so that $|A| \geq 4$.
By symmetry, $|B| \geq 4$, which implies $|A \cup B \cup C'| \geq 17$, a contradiction.
\end{proof}
\end{proof} 

\begin{lemma}
\label{l:app2part2}
Let $H$ be a graph, $v \in V(H)$ such that $H = N[v]$.
Suppose $\delta(H) \geq 10$ and $|H| = n \leq \min \{ 2 \delta(H) - 1, 19 \}$.
Then $H$ has a $(2,2,2,1)$-knitted subgraph.
\end{lemma}
\begin{proof}
Suppose $H$ has no $(2,2,2,1)$-knitted subgraph.
This implies that $H$ has no 7-clique. 
Since $H$ itself is not $(2,2,2,1)$-knitted, we will define $u_0, u_1, v_1, u_2, v_2, u_3, v_3$ as in the proof of Lemma~\ref{l:23case}, and define $C$ and $t$ as in the proof of Lemma~\ref{l:app2part1} except that we can allow $C_2$ to be a path on 5 vertices if no appropriate path on at most 4 vertices exists.
\begin{claim}
$t \geq 1$. 
\end{claim}
\begin{proof}
Suppose to the contrary that $t = 0$. 
Let $B_1 = N(v_1) - (C_0 \cup C_2 \cup C_3)$.
Then \[ |B_1| \geq \delta(H) - |C_0| - |C_2| - |C_3| = \delta(H) - 5 \geq 5. \]
Since $H[C_1]$ is disconnected, 
by minimality of $C$, $u_1$ can have no neighbor in $B_1 \cup \{ v \}$, and if any vertex of $B_1$ has a common neighbor with $u_1$, that common neighbor must belong to $C_0 \cup C_2 \cup C_3$.
If $B_1$ is a clique, then $B_1 \cup \{ v_1, v \}$ is a clique of order 7, a contradiction, so there exist $b, b' \in B_1$ that are not adjacent.
Then we have $N(u_1) \cup N(b) \subseteq V(H) - \{ u_1, b, b' \}$ and so \[ |N(u_1) \cap N(b)| = d(u_1) + d(b) - |N(u_1) \cup N(b)| \geq 2\delta(H) - (|H| - 3) \geq 4. \]
Since $N(u_1) \cap N(b) \subseteq \{ u_0, u_2, v_2, u_3, v_3 \}$ and $b$ has at least $|N(u_1) \cap N(b)| \geq 4$ neighbors in this set, $b$ must be complete to either $\{ u_2, v_2 \}$ or $\{ u_3, v_3 \}$, which implies that we can connect $C_2$ or $C_3$ using a path on 3 vertices, a contradiction.
\end{proof}

\begin{claim}
$t \geq 2$. 
\end{claim}
\begin{proof}
Suppose to the contrary that $t = 1$. 
Let $A_2 = N(u_2) - (C_0 \cup C_1 \cup C_3)$ and $B_2 = N(v_2) - (C_0 \cup C_1 \cup C_3)$.
Then we have \[ |N(u_2) \cap (C_0 \cup C_1 \cup C_3)| \leq |C_0 \cup C_1 \cup C_3| \leq 7, \] so $|A_2| \geq \delta(H) - 7 \geq 3$ and likewise $|B_2| \geq 3$.
By Lemma~\ref{l:si}(d), for any $a \in A_2$ and $b \in B_2$, we have \[ \sum_{i=0}^3 s_i \geq 3, \] where $s_0 \leq 1$ and $s_3 \leq 0$.
We also have $s_2 \leq 0$ by part (a) of Lemma~\ref{l:si}, 
so we must have $s_1 \geq 2$.
If $|C_1| = 4$, then this implies that every vertex in $A_2 \cup B_2$ has exactly 3 neighbors in $C_1$: by Lemma~\ref{l:si}(b), these 3 neighbors must be consecutive, so, if we write $C_1 = u_1 x y v_1$, then each vertex in $A_2 \cup B_2$ is complete to either $\{ u_1, x, y \}$ or $\{ x, y, v_1 \}$.
But this is impossible: since every vertex in $A_2 \cup B_2$ is necessarily complete to $\{ x, y \}$, if any $w \in A_2 \cup B_2$ is adjacent to $u_1$, we can perform an $(x, w)$-reroute of $C_1$ and $C_2$, and if $w$ is adjacent to $v_1$ instead, then we can perform a $(y, w)$-reroute of $C_1$ and $C_2$, with $C_2$ ending up as a 5-vertex path in either case.
Thus $|C_1| \leq 3$, which implies $|A_2| \geq \delta(H) - 6 \geq 4$ and likewise $|B_2| \geq 4$.
It follows from Lemma~\ref{l:si2}(c) that $A_2 \cup B_2$ is complete to $\{ u_1, v_1 \}$, and, if $|C_1| = 3$, it follows from the definition of $s_1$ that either every vertex from $A_2$ or every vertex from $B_2$ (without loss of generality, the former) is complete to $C_1$.
If, in the case where $|C_1| = 3$, any vertex from $B_2$ is complete to $C_1$ as well, then, if we label the middle vertex of $C_1$ as $x$, we can take any $a \in A_2$ and perform an $(x, a)$-reroute of $C_1$ and $C_2$, a contradiction.
Thus we cannot have $s_1 = 2$ for any choice of vertices in $A_2$ and $B_2$, so we have $\sum_{i=0}^3 s_i \leq 3$ and thus $\sum_{i=0}^3 s_i = 3$.
Lemma~\ref{l:si}(d) then implies that, for any $a \in A_2$ and $b \in B_2$, $|A_2 - N[a]| = |B_2 - N[b]| = 0$.
That is, $A_2$ and $B_2$ are cliques.
Now $C_1$ contains a $K_2$ that is complete to $A_2$, and $A_2$ is a clique on at least 4 vertices.
If this $K_2$ is complete to $u_2$ as well, then we get a 7-clique, which is impossible.
This implies $u_2$ has a non-neighbor in $C_1$, so that $|N(u_2) \cap (C_0 \cup C_1 \cup C_3)| \leq |C_0 \cup C_1 \cup C_3| - 1 \leq 5$ and thus $|A_2| \geq \delta(H) - 5 \geq 5$.
Then $A_2$ together with that $K_2$ is a 7-clique, a contradiction.
\end{proof}

Now we may assume $t=2$, so that $C_3$ is disconnected.
Let $A_3 = N(u_3) - C$ and $B_3 = N(v_3) - C$.

\begin{claim}
$A_3 \neq \varnothing$ and $B_3 \neq \varnothing$.
\end{claim}
\begin{proof}
Suppose not; without loss of generality, $A_3 = \varnothing$.
Then $N(u_3) \subseteq C_0 \cup C_1 \cup C_2$; since $|N(u_3)| \geq \delta(H) \geq 10$ and $|C_0 \cup C_1 \cup C_2| \leq 10$, we must have $|C_0 \cup C_1 \cup C_2| = 10$ (i.e., $|C_1| = 4$ and $|C_2| = 5$) with $u_3$ being complete to $C_0 \cup C_1 \cup C_2$.
Note that no interior vertex of $C_1$ or $C_2$ is adjacent to $v_3$, otherwise we can use that interior vertex to replace $C_3$ with a path of length 3, a shorter path than the one that we just disconnected, contrary to the minimality of $C$.
Write $C_2 = u_2 x_1 y x_2 v_2$.
Note that we have \[ |N(v_3) \cap N(y)| = d(v_3) + d(y) - |N(v_3) \cup N(y)| \geq 2\delta(H) - (|H| - 2) \geq 3. \]
Moreover, because $H - [N(v_3) \cup N(x_1)] \supseteq \{ v_3, x_1, x_2 \}$, we have \[ |N(v_3) \cap N(x_1)| \geq 2\delta(H) - (|H| - 3) \geq 4, \] and likewise $|N(v_3) \cap N(x_2)| \geq 4$.
Then each of $y, x_1, x_2$ has 3 neighbors outside of $C_2$ that are also neighbors of $v_3$ (of the four vertices in $N(v_3) \cap N(x_i)$, one is an endpoint of $C_2$).
None of these common neighbors can belong to $H - C$; otherwise, we could replace $C_3$ with a path of length 4 ($u_3$, an interior vertex of $C_2$, a common neighbor of that vertex and $v_3$, and $v_3$ itself) and replace $C_2$ with $\{ u_2, v_2 \}$ to get a choice of $C$ with fewer vertices, a contradiction.
Thus the three common neighbors for each pair must be the neighbors of $v_3$ in $C - C_2$, namely $u_0, u_1$, and $v_1$.
Since these are the only possible common neighbors, the inequalities we have above must be equalities: we have $|N(v_3) \cup N(y)| = |H| - 2$ and $|N(v_3) \cup N(x_i)| = |H| - 3$ for each $i \in [2]$, so that $N(v_3) \cup N(y) = H - \{ v_3, y \}$ and $N(v_3) \cup N(x_1) = N(v_3) \cup N(x_1) = H - \{ v_3, x_1, x_2 \}$.
That is, every vertex in $H - \{ v_3, x_1, y, x_2 \}$ is either adjacent to $v_3$ or complete to $\{ x_1, y, x_2 \}$; in particular, since the interior vertices of $C_1$ are anticomplete to $v_3$, they must be complete to $\{ x_1, y, x_2 \}$.
But then, if we write $C_1 = u_1 z_1 z_2 v_1$, we can replace $C_1$ with $u_1 y v_1$ and replace $C_2$ with $u_2 x_1 z_1 x_2 v_2$ to get a choice of $C$ with fewer vertices, a contradiction.
\end{proof}

\begin{claim}
$t=3$.
\end{claim}
\begin{proof}
Let $C' = C_0 \cup C_1 \cup C_2$.
For any $a \in A_3$, by Lemma~\ref{l:si}(b), we have $|N(a) \cap C_i| \leq 3$ for $i \in [2]$ and $|N(a) \cap C_0| \leq |C_0| = 1$, so that $|N(a) - C'| \geq \delta(H) - 1 - 3 - 3 \geq 3$.
That is, if we define $A$ and $B$ as in Lemma~\ref{l:si}, we have $|A| \geq 4$, and, by symmetry, $|B| \geq 4$.
Note that, since every $(A, B)$-path must pass through $C'$, the vertex $v$ that is complete to every other vertex in the graph must belong to $C'$; specifically, it must be either the sole vertex of $C_0$, an endpoint of a path $C_i$ such that $|C_i| = 2$, or the middle vertex of a path $C_i$ such that $|C_i| = 3$ (otherwise, there would be a vertex on $C_i$ that is adjacent to but not consecutive with $v$, meaning we would have a choice of $C_i$ with fewer vertices).
Observing that $|C_0| = 1$, $2 \leq |C_1| \leq 4$, and $2 \leq |C_1| \leq 5$, we proceed by cases according to the number of vertices in $C'$.
Throughout these cases, we will define $A_0 = \{ a \in A : A - N[a] \neq \varnothing \}$ and $B_0 = \{ b \in B : B - N[b] \neq \varnothing \}$.
\begin{case}
$|C'| = 5$.
\end{case}
\begin{proof}
Every vertex in $A$ has at most 5 neighbors in $C'$ and thus at least $\delta(H) - 5 \geq 5$ neighbors in $A$, so $|A| \geq 6$.
If $|A| = 6$, then, since $\delta(H[A]) \geq 5$, $A$ would be a 6-clique and so $A \cup \{ v \}$ would be a 7-clique, a contradiction, so we must have $|A| \geq 7$.
By symmetry, $|B| \geq 7$; since $|H| \leq 19 = 5 + 7 + 7$, we must have $|A| = |B| = 7$ exactly.
Then $A$ is a $K_7$ with a matching deleted.
More specifically, since $A$ contains no 6-clique, we must have $|A_0| \geq 4$.
Note that every vertex in $A_0$ is complete to $C'$ and every vertex in $A - A_0$ is adjacent to all but at most one vertex in $C'$.
If $|A_0| = 4$, then, since $|A - A_0| = 3$ and there are 4 vertices in $C' - v$, by the pigeonhole principle, some vertex $w \in C' - v$ must be complete to $A - A_0$ and thus complete to $A$.
Since $A$ contains a 5-clique, $A \cup \{ v, w \}$ contains a 7-clique, a contradiction.
Thus $|A_0| = 6$.
Let $w_1$ and $w_2$ be two adjacent vertices in $C'$ such that $v \notin \{ w_1, w_2 \}$, and let $a$ be the sole vertex of $A - A_0$; we may assume $a$ is adjacent to $w_2$.
We claim that $H[A \cup \{ w_1, w_2, v \}$ is $(2,2,2,1)$-knitted.
The non-adjacent pairs in this graph consist of three pairs of vertices in $A_0$ and possibly the pair $\{ a, w_1 \}$.
Since the complement of this graph has maximum degree 1, each non-adjacent pair has every other vertex in the graph as a common neighbor, for a total of $|A| + 3 - 2 = 8$ common neighbors.
Thus, by Corollary~\ref{c:common}, this graph is indeed $(2,2,2,1)$-knitted.
\end{proof}
\begin{case}
$|C'| = 6$.
\end{case}
\begin{proof}
We have $|C_1| = 2$ and $|C_2| = 3$, and every vertex in $A$ has at most 6 neighbors in $C'$ and thus at least $\delta(H) - 6 \geq 4$ neighbors in $A$.
Then $|A| \geq 5$; if $|A| = 5$, $A$ would be a 5-clique that was complete to $C'$, so, for any $w \in C' - v$, $A \cup \{ v, w \}$ would be a 7-clique, a contradiction, so we must have $|A| \geq 6$.
Since $A \cup \{ v \}$ is not a 7-clique, $A$ is not a 6-clique, so, since $\delta(H[A]) \geq 4$, $A$ is a $K_6$ with a matching deleted, and $|A_0| \geq 2$.
Note that every vertex in $A_0$ is complete to $C'$, and every vertex in $A - A_0$ is adjacent to every vertex in $C'$ except for at most one.

If $|A_0| = 2$, then each of the 4 vertices of $A - A_0$ is adjacent to 4 of the 5 vertices of $C' - v$, so, by the pigeonhole principle, there is $w \in C' - v$ that is complete to $A - A_0$ and thus complete to $A$.
Since $A$ contains a 5-clique, $A \cup \{ w, v \}$ contains a 7-clique, a contradiction.

If $|A_0| = 4$, we want to consider the graph $H[A \cup C_2 \cup \{ v \}]$.
We claim that $v \notin C_2$.
We know that $v \notin \{ u_2, v_2 \}$, otherwise we could replace $C_2$ with $u_2 v_2$ to get a choice of $C$ with fewer vertices.
For any $a, a' \in A_0$, $\{ a, a' \}$ is complete to $C_2$, so, if $v$ is the middle vertex of $C_2$, we could perform a $(v, a)$-reroute of $C_2$ and $C_3$, which is impossible.
Thus the graph $H[A \cup C_2 \cup \{ v \}]$ has 10 vertices; we claim that this graph is $(2,2,2,1)$-knitted.
The non-adjacent pairs in this graph are two pairs of vertices in $A_0$, the pair $\{ u_2, v_2 \}$, and up to two pairs that have one end in $A - A_0$ and one end in $C_2$.
The pair $\{ u_2, v_2 \}$ has 6 common neighbors in this graph: $v$, the middle vertex of $C_2$, and the four vertices in $A_0$.
Every non-adjacent pair with one end in $A - A_0$ and one end in $C_2$ has at least 7 common neighbors: the five other vertices in $A$, $v$, and at least one vertex in $C_2$ (recall that every vertex in $A - A_0$ has at most one non-neighbor in $C'$, so if it has one non-neighbor in $C_2$, it is adjacent to the other two vertices in $C_2$).
Every non-adjacent pair in $A_0$ has 8 common neighbors: every other vertex in the graph is a common neighbor.
If we choose three pairs of non-adjacent vertices in this graph, then the three vertices in $C_2$ can contribute to at most two of these pairs: either $\{ u_2, v_2 \}$ is one of the pairs, leaving only one vertex remaining in $C_2$, or else every pair with an end in $C_2$ has its other end in $A - A_0$, so there are at most two pairs.
We then have at most 2 pairs with at most 7 common neighbors, at most 1 of which has only 6 common neighbors, so, by Proposition~\ref{p:prop1}, this graph is $(2,2,2,1)$-knitted.

Thus we may assume $|A_0| = 6$, so that $A_0 = A$.
We now claim that the graph $H[A \cup C']$ is $(2,2,2,1)$-knitted.
Since $A = A_0$ is complete to $C'$, the non-adjacent pairs in this graph are 3 pairs of vertices in $A$ and some number of pairs of vertices in $C'$.
Every non-adjacent pair in $A$ has every other vertex in the graph, of which there are $|A| + |C'| - 2 = 10$, as common neighbors.
Every non-adjacent pair in $C'$ has at least 7 common neighbors: 6 vertices in $A$ as well as $v$.
Since $v$ is not part of any non-adjacent pair and $|C' - v| = 5$, if we choose three pairs of non-adjacent vertices in this graph, then at most two of these pairs have both ends in $C'$, so we have at most two pairs with 7 common neighbors, with every other pair having at least 10 common neighbors.
Thus, by Proposition~\ref{p:prop1}, this graph is $(2,2,2,1)$-knitted, a contradiction.
\end{proof}
\begin{case}
$|C'| = 7$.
\end{case}
\begin{proof}
Suppose $|C_1| = 2$, so that $|C_2| = 4$.
Then every vertex in $A - u_3$ has at most 1 neighbor in $C_0$, at most 2 neighbors in $C_1$, and at most 3 neighbors in $C_2$ by Lemma~\ref{l:si}(b), so every such vertex has at least $\delta(H) - 6 \geq 4$ neighbors in $A$.
Thus $|A| \geq 5$, and, by symmetry, $|B| \geq 5$.
If $|A| = 5$, then every vertex in $A$ except for at most 1 has 4 neighbors in $A$ and is thus complete to the rest of $A$, so $A$ must be a 5-clique.
Then every vertex in $A$ has exactly 4 neighbors in $A$ and thus exactly 6 neighbors in $C'$: every vertex in $A - u_3$ must then be complete to $C_0 \cup C_1$.
The vertex $u_3$ has at most one non-neighbor in $C'$, so it has at least two neighbors among the three vertices in $C_0 \cup C_1$, which means there is a vertex $w \in (C_0 \cup C_1) - v$ that is adjacent to $u_3$.
Then $A \cup \{ v, w \}$ is a 7-clique, a contradiction.
Therefore, $|A| \geq 6$, and, by symmetry, $|B| = 6$; since $|H| \leq 19$, we must have $|A| = |B| = 6$.
Since $v$ is complete to $A \cup B$, neither $A$ nor $B$ can be a 6-clique, so we can find $a, a' \in A$ and $b, b' \in B$ such that $aa', bb' \notin E(H)$; we may assume $a \neq u_3$ and $b \neq v_3$.
Applying Lemma~\ref{l:si}(d) to $a$ and $b$, we get $\sum_{i=0}^3 s_i \geq 5$.
Since $s_0 \leq |C_0| = 1$, $s_1 \leq |C_1| = 2$, and $s_3 \leq 0$ because $H[C_3]$ is disconnected, we have $s_2 \geq 2$.
For this to be possible, each of $a$ and $b$ must have exactly 3 neighbors in $C_2$.
Let $x$ be the interior vertex of $C_2$ that is in the middle of the three neighbors of $a$ on $C_2$; then $b$ is adjacent to $x$ as well, so $x$ can have no other neighbors in $A$, otherwise we can perform an $(x, a)$-reroute of $C_2$.
Thus no vertex of $A - a$ can have 3 consecutive neighbors on $C_2$, which implies that every vertex of $A - \{ a, u_3 \}$ has at most 2 neighbors on $C_2$ and thus has at most 5 total neighbors in $C'$.
Each of these neighbors then has 5 neighbors in $A$, making it complete to the rest of $A$; that is, the only possible pair of non-adjacent vertices in $A$ is the pair $\{ a, u_3 \}$.
Each of $a$ and $u_3$ then has 4 neighbors in $A$ and at most 3 neighbors in $C_2$, and each vertex of $A - \{ a, u_3 \}$ has 5 neighbors in $A$ and at most 2 neighbors in $C_2$, so that every vertex of $A$ has three neighbors in $C_0 \cup C_1$: that is, $A$ is complete to $C_0 \cup C_1$.
But then $A$ contains a 5-clique, so $A \cup C_1$ contains a 7-clique, a contradiction.

Thus $|C_1| \neq 2$; we must then have $|C_1| = |C_2| = 3$.
Every vertex in $A$ has at most 7 neighbors in $C'$, so $\delta(H[A]) \geq 3$ and thus $|A| \geq 4$.
If $|A| = 4$, then $A$ is a 4-clique and every vertex in $A$ has exactly 7 neighbors in $C'$; letting $w_1$ and $w_2$ be any two adjacent vertices in $C' - v$, the set $A \cup \{ v, w_1, w_2 \}$ must then be a 7-clique, a contradiction.
Thus $|A| \geq 5$, and, by symmetry, $|B| \geq 5$.

Suppose $|A| = 5$.
Then every vertex in $A$ has at most 4 neighbors in $A$, hence at least 6 neighbors in $C'$; in particular, every vertex of $A$ is complete to either $C_1$ or $C_2$.
We claim that $A$ has more than one vertex that is complete to $C_1$ and more than one vertex that is complete to $C_2$.
If not---if, say, $A$ has at most one vertex that is complete to $C_2$---choose $a \in A$ such that no vertex in $A - a$ is complete to $C_2$.
Then every vertex in $A - a$ has at most 6 neighbors in $C'$, hence at least 4 neighbors in $A$; that is, $A$ is a 5-clique, and at most one vertex of this 5-clique is not complete to $C_1$.
Note that the middle vertex of $C_1$ has no neighbor in $B$ (otherwise, we can reroute $C_1$ using any vertex in $A - u_3$), so $v \notin C_1$.
But then $A - a$, together with any two adjacent vertices on $C_1$ and with $v$, gives us a 7-clique.

Now at least two vertices of $A$, including some $a \in A - u_3$, are complete to $C_1$, and at least two vertices of $A$, including some $a' \in A - u_3$, are complete to $C_2$.
If any vertex of $B$ is adjacent to the middle vertex of $C_1$, we can reroute $C_1$ through $a$, and if any vertex of $B$ is adjacent to the middle vertex of $C_2$, we can reroute $C_2$ through $a'$.
Thus every vertex in $B$ has at most 1 neighbor in $C_0$, at most 2 neighbors in $C_1$, and at most 2 neighbors in $C_2$, for a total of at most 5 neighbors in $C'$.
Then every vertex in $B$ has at least 5 neighbors in $C'$; if $|B| = 6$, then $B$ is a 6-clique and so $B \cup \{ v \}$ is a 7-clique, a contradiction, so we must have $|B| \geq 7$, and the complement of $B$ is a matching, with every vertex in $B_0$ being complete to $\{ u_0, u_1, v_1, u_2, v_2 \}$.
In fact, since $u_1 v_1, u_2 v_2 \notin E(H)$ and since the middle vertices of $C_1$ and $C_2$ have no neighbor in $B$, it must be the case that $u_0 = v$.
If $|B_0| \leq 2$, then $B$ contains a 6-clique, so $B \cup \{ v \}$ contains a 7-clique, a contradiction.
We then have $|B_0| \geq 4$.
Every vertex in $B - B_0$ has 6 neighbors in $B$ and thus at least 4 neighbors in $C'$, so it has at least 3 neighbors in $\{ u_1, v_1, u_2, v_2 \}$.
We claim that there is a pair $\{ u_i, v_i \}$, $i \in [2]$, such that one of the two vertices in the pair is complete to $B - B_0$ and the other has at most 1 non-neighbor in $B - B_0$.
Since $|B - B_0| \leq 3$, at least one vertex in $\{ u_1, v_1, u_2, v_2 \}$ (without loss of generality, $u_1$) is complete to $B - B_0$.
If $v_1$ has at most 1 non-neighbor in $B - B_0$, then $\{ u_1, v_1 \}$ is our desired pair; if not, then there is at most one vertex in $B - B_0$ that is adjacent to $v_1$, and this vertex is the only vertex in $B - B_0$ that can have a non-neighbor in $\{ u_2, v_2 \}$, so $\{ u_2, v_2 \}$ is our desired pair.
Assume without loss of generality that $u_1$ is complete to $B - B_0$ and there is a vertex $b \in B - B_0$ such that $v_1$ is complete to $B - (B_0 \cup \{ b \})$.
Consider the graph $H[B \cup \{ u_1, v_1, v \}]$.
The non-adjacent pairs in this graph are two pairs that include $v_1$ (namely $\{ u_1, v_1 \}$ and $\{ v_1, b \}$), each of which has $|B| + 3 - 3 = 7$ common neighbors, and up to three pairs with both ends in $B_0$, each of which has $|B| + 3 - 2 = 8$ common neighbors.
By Corollary~\ref{c:uncommon}, this graph is $(2,2,2,1)$-knitted.

Now we have $|A| \geq 6$, and, by symmetry, $|B| \geq 6$; since $|A \cup B| \leq |H| - |C'| \leq 19 - 7 = 12$, we must have $|A| = |B| = 6$.
Since $v$ is complete to $A \cup B$, neither $A$ nor $B$ is a 6-clique, so $|A_0| \geq 2$ and $|B_0| \geq 2$.
We claim that there is $i \in [2]$ such that either two vertices of $A$ or two vertices of $B$ are complete to $C_i$.
Let $a \in A_0 - u_3$.
Then $a$ has at most 4 neighbors in $A$ and thus at least 6 neighbors in $C'$, so $a$ is complete to $C_1$ or $C_2$ (without loss of generality, the former).
Likewise, any $b \in B_0 - v_3$ must be complete to $C_1$ or $C_2$.
If $b$ is complete to $C_1$, then the middle vertex $x_1$ of $C_1$ has no neighbor in $(A \cup B) - \{ a, b \}$, otherwise we can perform an $(x_1, a)$- or $(x_1, b)$-reroute of $C_1$ and $C_3$.
Thus $b$ is complete to $C_2$.
Each of $u_3$ and $v_3$ also has at least 6 neighbors in $C'$ and is thus complete to $C_1$ or $C_2$.
If $u_3$ is complete to $C_1$ or $v_3$ is complete to $C_2$, we are done, so we may assume $u_3$ is complete to $C_2$, $v_3$ is complete to $C_1$, and no vertex from $(A \cup B) - \{ u_3, v_3, a, b \}$ is complete to $C_1$ or to $C_2$.
Then every vertex in $A - \{ u_3, a \}$ has at most 5 neighbors in $C_1$, hence at least 5 neighbors in $A$; that is, every vertex in $A - \{ u_3, a \}$ is adjacent to every other vertex in $A$ and has exactly 5 neighbors in $C'$ (1 in $C_0$, 2 in $C_1$, 2 in $C_2$).
Note that no vertex in $A - u_3$ is adjacent to the middle vertex $x_1$ of $C_1$, otherwise we can perform an $(x_1, a)$-reroute of $C_1$.
This implies that $A - u_3$ is complete to $\{ u_1, v_1 \}$.
But then, for any $a' \in A - u_3$, we can perform an $(x_1, a')$-reroute of $C_1$.

Now we have some $C_i$ that is complete to either two vertices in $A$ or two vertices in $B$; we may assume $C_1$ is complete to two vertices in $A$.
Let $a \in A - u_3$ be complete to $C_1$.
Then the middle vertex $x_1$ of $C_1$ has no neighbor in $B$, so every vertex of $B_0$ (including $v_3$) must be complete to $C_2 \cup \{ u_0, u_1, v_1 \}$.
This, in turn, implies that the middle vertex $x_2$ of $C_2$ has no neighbor in $A$, so it must be the case that $u_0 = v$.
Every vertex of $B - B_0$ has at least 5 neighbors in $C'$, which necessarily include $v$ but not $x_1$, so each one is adjacent to at least 4 of the 5 vertices in $C_2 \cup \{ u_1, v_1 \}$.
Since $|B - B_0| \leq 4$, this implies that some vertex in $C_2 \cup \{ u_1, v_1 \}$ is complete to $B - B_0$ and thus complete to $B$.
If $|B_0| = 2$, then $B$ contains a 5-clique, so taking a vertex in $C_2 \cup \{ u_1, v_1 \}$ that is complete to $B$, together with $v$ and the 5-clique in $B$, gives us a 7-clique, a contradiction.
Thus, since no vertex in $B$ can have more than 1 non-neighbor in $B$ (each vertex of $B$ has at most 6 neighbors in $C'$ and thus at least 4 neighbors in $B$), we must have $|B_0| \in \{ 4, 6 \}$.
Consider the graph $H[B \cup C_2 \cup \{ v \}]$.
This graph has 10 vertices, and the non-adjacent pairs of vertices in this graph consist of 2 or 3 pairs of vertices in $B_0$, the pair $\{ u_2, v_2 \}$, and up to 2 pairs with one end in $B - B_0$ and the other end in $C_2$.
If $|B_0| = 6$, then $B - B_0$ is empty, so every vertex in this graph has at most 1 non-neighbor, which means that every non-adjacent pair has 8 common neighbors.
Thus, by Corollary~\ref{c:common}, this graph is $(2,2,2,1)$-knitted if $|B_0| = 6$, so we must have $|B_0| = 4$.
If the 2 vertices of $B - B_0$ are complete to 2 adjacent vertices in $C_2$, then those 2 vertices, together with a 4-clique in $B$ and with $v$, give us a 7-clique, a contradiction, so the 2 vertices of $B - B_0$ are either anticomplete to $x_2$ or else each one has a different non-neighbor in $C_2$.
Either way, each non-adjacent pair in $B_0$ has 8 common neighbors.
If the 2 vertices of $B - B_0$ are anticomplete to $x_2$, then any non-adjacent pair including $x_2$ has 7 common neighbors (every vertex in the graph except for $x_2$ and its 2 non-neighbors) and the pair $\{ u_2, v_2 \}$ has 8 common neighbors, so the graph is $(2,2,2,1)$-knitted by Corollary~\ref{c:uncommon}.
Thus we may assume that each of the 2 vertices of $B - B_0$ has a different non-neighbor in $C_2$: we will write $B - B_0 = \{ b_1, b_2 \}$ and assume without loss of generality that $b_1$ is not adjacent to $u_2$.
Then the non-neighbor of $b_2$ is either $x_2$ or $v_2$.
If the non-neighbor of $b_2$ is $x_2$, then each non-adjacent pair including $u_2$ has 7 common neighbors (every vertex except $u_2$, $b_1$, and $v_2$), and each other non-adjacent pair (the two pairs in $B_0$ as well as $\{ v_2, x_2 \}$) has 8 common neighbors, so the graph is $(2,2,2,1)$-knitted by Corollary~\ref{c:uncommon}.
If the non-neighbor of $b_2$ is $v_2$, then the pair $\{ u_2, v_2 \}$ has 6 common neighbors (the 4 vertices in $B_0$ as well as $x_2$ and $v$), the pair $\{ b_1, u_2 \}$ has 7 common neighbors (the 4 vertices in $B_0$ as well as $x_2, v$, and $b_2$), the pair $\{ b_2, v_2 \}$ has 7 common neighbors as well, and every pair in $B_0$ has 8 common neighbors.
So, given any three disjoint pairs of non-adjacent vertices in this graph, the given pairs can include the pair with 6 common neighbors or at least one of the pairs with 7 common neighbors, but cannot include all three since they overlap.
With all other pairs having 8 common neighbors, Proposition~\ref{p:prop1} shows that this graph is $(2,2,2,1)$-knitted.
\end{proof}
\begin{case}
$|C'| = 8$.
\end{case}
\begin{proof}
Suppose $|C_1| = 2$; then $|C_2| = 5$.
Every vertex in $A - u_3$ then has at most 6 neighbors in $C'$, hence at least 4 neighbors in $A$, so $|A| \geq 5$, and, by symmetry, $|B| \geq 5$.
We have $|A \cup B| \leq |H| - |C'| \leq 19 - 8 = 11$, so we must have $\min\{ |A|, |B| \} = 5$; without loss of generality, $|A| = 5$.
But then every vertex in $A$ has at most 4 neighbors in $A$ and thus at least 6 neighbors in $C'$: in particular, each of the four vertices of $A - u_3$ is complete to $C_0 \cup C_1$.
If there is $w \in (C_0 \cup C_1) - v$ that is adjacent to $u_3$, then $A \cup \{ v, w \}$ is a 7-clique, so $u_3$ must have at most 1 neighbor in $C_0 \cup C_1$, which implies that $u_3$ must be complete to $C_2$.
Then $v_3$ is not adjacent to any interior vertex of $C_2$, otherwise we can use that interior vertex to turn $C_3$ into a path on 3 vertices, fewer vertices than $C_2$, contrary to the minimality of $C$.
Thus $v_3$ has at most 2 neighbors in $C_2$, which implies that it has at most 5 neighbors in $C'$ and thus at least 5 neighbors in $B$; that is, $|B| = 6$ and every vertex in $B - v_3$ is adjacent to $v_3$.
But then every vertex of $B - v_3$ has at most 5 neighbors in $B$ and at most 3 neighbors in $C_0 \cup C_1$, so it has at least 2 neighbors in $C_2$; by Lemma~\ref{l:si}(b), these 2 neighbors cannot be the endpoints of $C_2$, so one of them is an interior vertex of $C_2$, which means we can connect $u_3$ to $v_3$ with a path of length 4, again contradicting the minimality of $C$.

Thus we must have $|C_1| \geq 3$, which implies $|C_1| = 3$ and $|C_2| = 4$.
Now every vertex in $A - u_3$ has at most 7 neighbors in $C'$, hence at least 3 neighbors in $A$, so that $|A| \geq 4$, and, if $|A| = 4$, $A$ is a 4-clique.
In that case, every vertex of $A - u_3$ is complete to $C_0 \cup C_1$; the middle vertex of $C_1$ then has no neighbor in $B$ (otherwise, we could reroute $C_1$ using any vertex in $A - u_3$), so $v \notin C_1$, which implies that $v$ is the sole vertex of $C_0$.
Moreover, every vertex of $A - u_3$ has 3 consecutive neighbors in $C_2$ and thus is complete to the middle 2 vertices of $C_2$.
Since $u_3$ also has at most 3 neighbors in $A$ and thus at least 7 neighbors in $C'$, it is either complete to the middle 2 vertices of $C_2$ or else it is complete to $C_1$.
Either way, there are two adjacent vertices in $C_1$ or in $C_2$ that are complete to $A$, and these six vertices, together with $v$, give us a 7-clique.
Thus $|A| \geq 5$, and, by symmetry, $|B| \geq 5$.
We have $|A \cup B| \leq |H| - |C'| \leq 19 - 8 = 11$, so we must have $\min\{ |A|, |B| \} = 5$; without loss of generality, $|A| = 5$.

Suppose $A$ is a 5-clique. 
Then every vertex in $A$ has at most 4 neighbors in $A$ and thus at least 6 neighbors in $C'$, so it must have 3 consecutive neighbors in either $C_1$ or $C_2$.
By the pigeonhole principle, there is a set of 3 consecutive vertices in $C_1$ or in $C_2$ that is complete to 2 vertices in $A$.
The middle vertex $x$ of this set of 3 vertices then has no neighbor in $B$; otherwise, taking $a \in A - u_3$ that is complete to the set of 3 vertices, we can perform an $(x, a)$-reroute of that path and $C_3$.
Then no vertex in $B$ can have 3 consecutive neighbors in that $C_i$, so every vertex in $B - v_3$ has at most 2 neighbors in that $C_i$ and thus at least 3 neighbors in either $C_1$ or $C_2$ (whichever one does not contain $x$).
Letting $y$ be the middle vertex of the 3 neighbors of any vertex in $B - v_3$ in this $C_j$, a similar argument shows that $y$ has no neighbor in $A$, so every vertex in $A - u_3$ has at most 2 neighbors in $C_j$ and thus at least 3 neighbors in $C_i$.
That is, either $A - u_3$ or $B - v_3$ is complete to $C_1$, and $C_1$ does not contain $v$.
If it is $A - u_3$ that is complete to $C_1$, then $A - u_3$, together with $v$ and with any two adjacent vertices on $C_1$, would be a 7-clique, a contradiction.
Thus $B - v_3$ must be complete to $C_1$, and so every vertex of $A - u_3$ must have 3 consecutive neighbors in $C_2$.
But then $A - u_3$ is complete to the middle two vertices of $C_2$, and these six vertices are complete to $v$, again giving us a 7-clique.

Thus $A$ is not a 5-clique, and, by symmetry, $B$ is not a 5-clique.
Let $a \in A_0 - u_3$ and $b \in B_0 - v_3$.
Then $a$ has at most 3 neighbors in $A$ and thus exactly 7 neighbors in $C'$: it is complete to $C_0 \cup C_1$ and has 3 consecutive neighbors in $C_2$.
If $|B| = 5$, then $b$ is likewise complete to $C_0 \cup C_1$.
But then, letting $x$ be the middle vertex of $C_1$, $x$ has no neighbor in $(A \cup B) - \{ u_3, v_3, a, b \}$, otherwise we can perform an $(x, a)$- or $(x, b)$-reroute of $C_1$ and $C_3$.
So $x$ has 2 neighbors in $A \cup B$, 2 neighbors in $C_1$, at most 1 neighbor in $C_0$, and at most 4 neighbors in $C_2$, for a total of $9 < \delta(H)$ neighbors, a contradiction (note that $V(H) = A \cup B \cup C$, as any vertex outside of $A \cup B$ must have all of its neighbors in $C$ by definition, but $|C| < \delta(H)$, so this is impossible).
Thus $|B| = 6$; now $b$ must have exactly 4 neighbors in $B$ and exactly 6 neighbors in $C'$, comprising 1 neighbor in $C_0$, 2 neighbors in $C_1$, and 3 neighbors in $C_2$.
If $a$ and $b$ have the same 3 neighbors in $C_2$, counting the neighbors of the middle vertex of those 3 common neighbors gives us the same contradiction, so, if we write $C_2 = u_2 x y v_2$, we may assume $a$ is complete to $\{ u_2, x, y \}$ and $b$ is complete to $\{ x, y, v_2 \}$.
Then $x$ has no neighbor in $A - a$, otherwise we can perform an $(x, a)$-reroute of $C_2$ and $C_3$, and $y$ has no neighbor in $B - b$, otherwise we can perform a $(y, b)$-reroute of $C_2$ and $C_3$.
Then every vertex in $B - \{ b, v_3 \}$ has at most 2 neighbors in $C_2$ and at most 2 neighbors in $C_3$, thus each one must have exactly 5 neighbors in $C'$ ($u_0, u_1, v_1$, $x$, and either $u_2$ or $v_2$) and exactly 5 neighbors in $B$; that is, the sole non-neighbor of $b$ in $B$ must be $v_3$.
Moreover, each of $b$ and $v_3$ must have exactly 4 neighbors in $B$ and exactly 6 neighbors in $C'$ (both are complete to $\{ u_0, u_1, v_1, x, v_2 \}$, $b$ is adjacent to $y$, and $v$ is adjacent to $u_2$).
Note that, because $u_1 v_1 \notin E(H)$ and the middle vertex of $C_1$ has no neighbor in $B$, $v \notin C_1$ and thus $v = u_0$.
But then $(B - v_3) \cup \{ u_0, u_1 \}$ is a 7-clique, a contradiction.
\end{proof}
\begin{case}
$|C'| = 9$.
\end{case}
\begin{proof}
Since $|C_0| = 1$ and $|C_2| \leq 5$, we must have $|C_1| \geq 3$.
Suppose $|C_1| = 3$.
Then every vertex in $A - u_3$ has at most 7 neighbors in $C'$ and thus at least 3 neighbors in $A$.
If $|A| = 4$, then $A$ is a 4-clique, and every vertex in $A - u_3$ has exactly 7 neighbors in $C'$; in particular, $A - u_3$ is complete to $C_0 \cup C_1$.
Then the middle vertex of $C_1$ is anticomplete to $B$ (otherwise, we can reroute $C_1$ through any vertex of $A - u_3$), so $v \notin C_1$.
If $u_3$ has 2 adjacent neighbors in $C_1$, then those 2 neighbors, together with $A$ and $v$, form a 7-clique, so $u_3$ must not be adjacent to the middle vertex of $C_1$.
Then $u_3$ has at most 3 neighbors in $C_0 \cup C_1$ and at most 3 neighbors in $A$, so it has at least 4 neighbors in $C_2$; in particular, at least 2 of the interior vertices of $C_2$ are adjacent to $u_3$.
Note that no interior vertex of $C_2$ is adjacent to both $u_3$ and a vertex in $B \cap N[v_3]$; if it were, then it would allow us to connect $u_3$ to $v_3$ with a path on at most 4 vertices, which would be shorter than $C_2$, contrary to the minimality of $C$.
This implies that every vertex in $B \cap N(v_3)$ has at most 1 interior neighbor in $C_2$; applying Lemma~\ref{l:si}(b), we can see that every vertex in $B \cap N(v_3)$ must then have at most 2 total neighbors in $C_2$, at most 1 neighbor in $C_0$, and at most 2 neighbors in $C_1$, hence at least 5 neighbors in $B$.
We then have $6 \leq |B| \leq |H| - |A| - |C'| \leq 19 - 4 - 9 = 6$, so $|B| = 6$ exactly, and every vertex in $B$ that is adjacent to $v_3$ is complete to $B$.
The vertex $v_3$ has at most 3 neighbors in $C_2$ (the two endpoints and at most 1 interior vertex of $C_2$ that is not adjacent to $u_3$), at most 2 neighbors in $C_1$, and at most 1 neighbor in $C_0$, so it has at least 4 neighbors in $B$.
Since $B$ cannot be a 6-clique (otherwise, $B \cup \{ v \}$ would be a 7-clique), $v_3$ must have exactly 4 neighbors in $B$, so $H[B]$ is a $K_6$ with a single edge (call it $v_3 b$) deleted.
The vertex $b$ has exactly 4 neighbors in $B$ and at most 3 neighbors in $C_0 \cup C_1$, so it has exactly 3 neighbors in $C_2$.
Recall that every vertex in $A - u_3$ has exactly 7 neighbors in $C'$ and thus exactly 3 neighbors in $C_2$.
If we write $C_2 = u_2 x y z v_2$, then this implies that $A - u_3$ is complete to $y$; this, in turn, implies that no $a \in A - u_3$ is complete to $\{ x, y, z \}$, otherwise we can perform a $(y, a)$-reroute of $C_2$ and $C_3$.
Then every vertex of $A - u_3$ is complete to either $\{ u_2, x, y \}$ or $\{ y, z, v_2 \}$; we may assume without loss of generality that at least 2 vertices of $A - u_3$ are complete to $\{ u_2, x, y \}$.
Then $x$ has no neighbor in $B$, so the neighbors of $b$ on $C_2$ must then be $y, z, v_2$.
But then, if any $a \in A - u_3$ is complete to $\{ y, z, v_2 \}$, the vertex $z$ must have no neighbor in $(A \cup B) - \{ a, b \}$, otherwise we can perform a $(z, a)$- or $(z, b)$-reroute of $C_2$ and $C_3$.
Then $z$ has 2 neighbors in $A \cup B$, 2 neighbors in $C_2$, and at most 4 neighbors in $C_0 \cup C_1$ (note that $|H| \leq 19 = |C'| + |A| + |B|$, so $V(H) = A \cup B \cup C'$), so $d(z) \leq 8 < \delta(H)$, a contradiction.
Thus $A - u_3$ is complete to $\{ u_2, x, y \}$.
If $u_3$ is complete to 2 adjacent vertices in the set $\{ u_2, x, y \}$, then those 2 vertices, together with $A$ and $v$, give us a 7-clique.
If not, then the neighbors of $u_3$ in $C_2$, of which there are at least 4, must be exactly $u_2, y, z, v_2$.
As previously observed, no vertex in $B \cap N[v]$ can be adjacent to $y$ or $z$.
The vertex $z$ then has 2 neighbors in $C_2$, at most 4 neighbors in $C_0 \cup C_1$, 1 neighbor in $A$ (namely $u_3$), and 1 neighbor in $B$ (namely $b$), so $d(z) \leq 8 < \delta(H)$, a contradiction.

Thus we may assume $|A| \geq 5$, and, by symmetry, $|B| \geq 5$.
Since $|A \cup B| \leq |H| - |C'| \leq 19 - 9 = 10$, we must then have $|A| = |B| = 5$.
Every vertex in $A$ then has at most 4 neighbors in $A$ and thus at least 6 neighbors in $C'$; the same is true for $B$.
Suppose some $a \in A - u_3$ and some $b \in B - v_3$ are both complete to $C_1$.
Then the middle vertex $x$ of $C_1$ has no neighbor in $(A \cup B) - \{ u_3, v_3, a, b \}$, otherwise we can perform an $(x, a)$- or $(x, b)$-reroute of $C_1$ and $C_3$.
Since $V(H) = A \cup B \cup C'$ and $x$ has at most 2 neighbors in $A \cup B$, 2 neighbors in $C_1$, and at most 1 neighbor in $C_0$, it must have exactly 5 neighbors in $C_2$; that is, it must be complete to $C_2$.
But then we can replace $C_1$ with $u_1 a v_1$ and replace $C_2$ with $u_2 x v_2$ to get a choice of $C$ with fewer vertices, contrary to the minimality of $C$.
Thus, we may assume without loss of generality that no vertex in $A - u_3$ is complete to $C_1$, which implies that every vertex in $A - u_3$ has exactly 1 neighbor in $C_0$, 2 neighbors in $C_1$, 3 neighbors in $C_2$, and 4 neighbors in $A$.

We claim that a vertex in $B - v_3$ has 3 neighbors in $C_2$.
If not, then every vertex in $B - v_3$ must have 2 neighbors in $C_2$, 1 neighbor in $C_0$, 3 neighbors in $C_1$, and 4 neighbors in $B$, so that $B$ is a 5-clique.
Then the middle vertex $x$ of $C_1$ has no neighbor in $A$ (otherwise, for any $b \in B - v_3$, we can perform an $(x, b)$-reroute of $C_1$ and $C_3$), so $v$ must be the sole vertex of $C_0$.
But then taking any 2 adjacent vertices on $C_1$, together with $v$ and the 4 vertices of $B - v_3$, gives us a 7-clique, a contradiction.

Let $a \in A - u_3$ and $b \in B - v_3$ each have 3 neighbors in $C_2$.
If we write $C_2 = u_2 x y z v_2$, then we claim that no vertex in $(A \cup B) - \{ u_3, v_3 \}$ is complete to $\{ x, y, z \}$.
Since every vertex in $A - u_3$ has 3 consecutive neighbors in $C_2$, $A - u_3$ is complete to $y$, so, if some $a' \in A - u_3$ is complete to $\{ x, y, z \}$, we can perform a $(y, a')$-reroute of $C_2$ and $C_3$.
Then every vertex in $A - u_3$ is complete to either $\{ u_2, x, y \}$ or $\{ y, z, v_2 \}$, so we either have at least 2 vertices complete to $\{ u_2, x, y \}$, in which case $x$ is anticomplete to $B$, or we have at least 2 vertices complete to $\{ y, z, v_2 \}$, in which case $z$ is anticomplete to $B$; either way, no vertex of $B$ can be complete to $\{ x, y, z \}$.
If $a$ and $b$ are complete to the same 3 vertices in $C_2$, say $\{ u_2, x, y \}$, then $x$ can have no neighbor in $(A \cup B) - \{ a, b \}$, otherwise we can perform an $(x, a)$- or $(x, b)$-reroute of $C_2$ and $C_3$.
So $x$ has 2 neighbors in $A \cup B$, 2 neighbors in $C_2$, and at most 4 neighbors in $C_0 \cup C_1$ (note that, because $|H| \leq 19 = |A| + |B| + |C'|$, we have $V(H) = A \cup B \cup C'$), so $d(x) = 8 < \delta(H)$, a contradiction.
Thus we may assume without loss of generality that $a$ is complete to $\{ u_2, x, y \}$ and $b$ is complete to $\{ y, z, v_2 \}$.
This same argument shows that no vertex of $A$ can be complete to $\{ y, z, v_2 \}$; since we have shown that every vertex in $A - u_3$ has 3 neighbors in $C_2$, it follows that $A - u_3$ is complete to $\{ u_2, x, y \}$.
Moreover, every vertex in $A - u_3$ has exactly 6 neighbors in $C'$ and thus exactly 4 neighbors in $A$, so $A$ is a 5-clique.
But then $A \cup \{ u_2, x, v \}$ is a 7-clique, a contradiction.

Now we may assume $|C_1| \neq 3$; this implies $|C_1| = |C_2| = 4$.
As before, every vertex in $A$ has at most 7 neighbors in $C'$ and thus at least 3 neighbors in $A$.
This implies that every vertex in $A - u_3$ has 3 neighbors each in $C_1$ and in $C_2$.
Let $a, a' \in A - u_3$; let $x$ be the middle vertex of the 3 neighbors of $a$ on $C_1$, and let $y$ be the middle vertex of the 3 neighbors of $a$ on $C_2$.
Then $a'$ is adjacent to both $x$ and $y$ as well, so $\{ x, y \}$ is anticomplete to $B$, otherwise we can perform an $(x, a)$-reroute of $C_1$ and $C_3$ or else a $(y, a)$-reroute of $C_2$ and $C_3$.
Then every vertex in $B - v_3$ has at most 2 neighbors in $C_1$, at most 2 neighbors in $C_2$, and at most 1 neighbor in $C_0$, so it has at least 5 neighbors in $B$.
Since $|B| \leq |H| - |A| - |C'| \leq 19 - 4 - 9 = 6$, $B$ must then be a 6-clique, so that $B \cup \{ v \}$ is a 7-clique, a contradiction.

Now we must have $|A| = |B| = 5$.
Every vertex in $A$ has at most 4 neighbors in $A$ and thus at least 6 neighbors in $C'$, so it has 3 neighbors in $C_1$ or 3 neighbors in $C_2$.
By the pigeonhole principle, we may assume that 2 vertices of $A - u_3$ each have 3 neighbors in $C_1$; if we write $C_1 = u_1 x_1 y_1 v_1$, we may assume some $a \in A - u_3$ is complete to $\{ u_1, x_1, y_1 \}$.
Then $x_1$ has no neighbor in $B$, otherwise we can perform an $(x_1, a)$-reroute of $C_1$ and $C_3$.
This implies that every vertex of $B - v_3$ must have at most 2 neighbors in $C_1$ and thus at least 3 neighbors in $C_2$.
Since every vertex of $B - v_3$ then has at most 6 neighbors in $C'$, each one has at least 4 neighbors in $B$, so $B$ is a 5-clique.
Moreover, every vertex of $B - v_3$ must be complete to the same 3 vertices in $C_2$: if we write $C_2 = u_2 x_2 y_2 v_2$ and assume that some $b \in B - v_3$ is complete to $\{ u_2, x_2, y_2 \}$ and some $b' \in B - v_3$ is complete to $\{ x_2, y_2, v_2 \}$, then neither $x_2$ nor $y_2$ can have a neighbor in $A$, otherwise we can perform an $(x_2, b)$- or $(y_2, b')$-reroute of $C_2$ and $C_3$.
But then every vertex of $A$ would have at most 1 neighbor in $C_2$, contrary to the fact that every vertex in $A$ has at least 6 neighbors in $C'$.
Thus, we may assume $B - v_3$ is complete to $\{ u_2, x_2, y_2 \}$.
The vertex $v_3$ has exactly 4 neighbors in $B$, at most 1 neighbor in $C_0$, and at most 3 neighbors in $C_1$, so it must have at least 2 neighbors in $C_2$, which means it has a neighbor in $\{ u_2, x_2, y_2 \}$.
This neighbor, together with $B$ and $v$, gives us a 7-clique.
\end{proof}
\begin{case}
$|C'| = 10$.
\end{case}
\begin{proof}
We have $|C_1| = 4, |C_2| = 5$; note that $v$ must be the sole vertex of $C_0$.
We also have $|A \cup B| \leq |H| - |C'| \leq 19 - 10 = 9$, so we may assume $|A| = 4$ and $|B| \in \{ 4, 5 \}$.
By Lemma~\ref{l:si}(b) every vertex in $A - u_4$ has at most 7 neighbors in $C'$ (1 in $C_0$ and 3 each in $C_1$ and $C_2$), so it has at least $\delta(H) - 7 \geq 3$ neighbors in $A$, so that $A$ is a 4-clique.
This implies that every vertex in $A$ has exactly 3 neighbors in $A$ and thus exactly 7 neighbors in $C'$.
If we write $C_1 = u_1 x_1 x_2 v_1$, then every vertex in $A - u_3$ is complete to either $\{ u_1, x_1, x_2 \}$ or $\{ x_1, x_2, v_1 \}$; we may assume there is $a \in A$ that is complete to $\{ u_1, x_1, x_2 \}$.
Then $x_1$ has no neighbor in $B$, otherwise we can perform an $(x_1, a)$-reroute of $C_1$ and $C_3$.
This means that no vertex in $B - v_3$ has 3 consecutive neighbors on $C_1$, hence each such vertex has at most 2 neighbors total on $C_1$ and at most 6 neighbors in $C'$.
Thus each vertex in $B - v_3$ has 4 neighbors in $B$, so $B$ must be a 5-clique: every vertex in $B - v_3$ has exactly 4 neighbors in $B$ and thus exactly 5 neighbors in $C'$, which must be 1 neighbor in $C_0$, 2 in $C_1$, and 3 in $C_2$.
This implies that $x_2$ is complete to $B - v_3$, which, in turn, implies that no $a \in A - u_3$ is complete to $\{ x_1, x_2, v_1 \}$, otherwise we could perform an $(x_2, a)$-reroute of $C_1$ and $C_3$.
If $u_3$ is adjacent to $x_1$, then $A \cup \{ x_1, x_2, v \}$ would be a 7-clique, a contradiction.
Thus $u_3$ is not adjacent to $x_1$, so that $u_3$ has at most 3 neighbors in $C_1$ and thus at least 3 neighbors in $C_2$.

Write $C_2 = u_2 y_1 y_2 y_3 v_2$.
Then every vertex of $(A \cup B) - \{ u_3, v_3 \}$ is complete to either $\{ u_2, y_1, y_2 \}$, $\{ y_1, y_2, y_3 \}$, or $\{ y_2, y_3, v_2 \}$.
Since $y_2$ is complete to $(A \cup B) - \{ u_3, v_3 \}$, no vertex $w \in (A \cup B) - \{ u_3, v_3 \}$ is complete to $\{ y_1, y_2, y_3 \}$, otherwise we could perform a $(y_2, w)$-reroute of $C_2$ and $C_3$.
We assume without loss of generality that two vertices in $A - u_3$ are complete to $\{ u_2, y_1, y_2 \}$.
Then $y_1$ has no neighbor in $B$: otherwise, for any $a \in A - u_3$ that is complete to $\{ u_2, y_1, y_2 \}$, we can perform a $(y_1, a)$-reroute of $C_2$ and $C_3$.
It follows that every vertex in $B - v_3$ is complete to $\{ y_2, y_3, v_2 \}$, so that $y_3$ is anticomplete to $A$ and so every vertex in $A - u_3$ is complete to $\{ u_2, y_1, y_2 \}$.
In particular, $u_3$ is not adjacent to $y_3$; moreover, $u_3$ is not adjacent to $y_2$, otherwise, for any $b \in B - b_3$, we could replace $C_2$ with $\{ u_2, v_2 \}$ and replace $C_3$ with $u_3 y_2 b v_3$ to get a choice of $C$ with fewer vertices.
Therefore, the three neighbors of $u_3$ on $C_3$ are $\{ u_2, y_1, v_2 \}$.
This implies that $A \cup \{ u_2, y_1, v \}$ is a 7-clique, a contradiction.
\end{proof}
\end{proof} 
\end{proof} 

\end{document}